\documentclass[12pt]{amsart}
\usepackage{epic,latexsym,amsbsy,amsmath,amscd,amssymb}
\usepackage{graphicx}

\newcommand\Acal{\mathcal{A}}%{\operatorname{\sf A}}
\newcommand\AND{\quad\text{and}\quad}

\newcommand\bd{\partial}

\newcommand\Ccal{\mathcal C}

\newcommand\de{\delta}
\newcommand\diam{\operatorname{\rm diam}}

\newcommand\EE{E}  %\mathsf{E}}
\newcommand\ep{\epsilon}
\newcommand\F{\mathbb F}

\newcommand\FF{F}  %\mathsf{F}}

\newcommand\impliess{\buildrel * \over \implies}
\newcommand\K{\mathbb K}
\newcommand\Pb{\mathbf{P}}
\newcommand\Pcal{\mathcal P}

\newcommand\PP{\mathsf{P}}

\newcommand\Q{\mathbf{Q}}

\newcommand\Si{\mathbf{\Sigma}}
\newcommand\T{\mathbb T}

\newcommand\then{\!\!\implies\!\!}
\newcommand\thens{\!\!\impliess\!\!}

\newcommand\wt{\widetilde}
\newcommand\V{\mathbf{V}}
\newcommand\VV{X}   %\mathsf{V}}

\newcommand\XX{X}   %\mathsf{X}}

\newcommand\Z{\mathbf{Z}}
\newcommand\ZZ{\mathbb{Z}}

\numberwithin{equation}{section}

\newtheoremstyle{mythm}% name
  {9pt}%      Space above, empty = `usual value'
  {9pt}%      Space below
  {\itshape}% Body font
  {0pt}%         Indent amount (empty = no indent, \parindent = paraindent)
  {\bfseries}% Thm head font
  {}%        Punctuation after thm head
  { }% Space after thm head: \newline = linebreak
  {\thmnumber{(#2)}\thmname{ #1}\thmnote{ #3}}%         Thm head spec

\newtheoremstyle{mydef}% name
  {9pt}%      Space above, empty = `usual value'
  {9pt}%      Space below
  {\normalfont}% Body font
  {0pt}%         Indensf suppt amount (empty = no indent, \parindent = paraindent)
  {\bfseries}% Thm head font
  {}%        Punctuation after thm head
  { }% Space after thm head: \newline = linebreak
  {\thmnumber{(#2)}\thmname{ #1}\thmnote{ #3}}%         Thm head spec

\theoremstyle{mythm}
\newtheorem{thm}[equation]{Theorem.}
\newtheorem{pro}[equation]{Proposition.}
\newtheorem{lem}[equation]{Lemma.}
\newtheorem{cor}[equation]{Corollary.}
\newtheorem{dlm}[equation]{Lemma/Definition.}

\theoremstyle{mydef}
\newtheorem{dfn}[equation]{Definition.}
\newtheorem{exa}[equation]{Example.}

\newtheorem{rmk}[equation]{Remark.}

\begin{document}$\,$ \vspace{-1truecm}
\title{Context-free pairs of groups \\ 
I - Context-free pairs and graphs}
\author{\bf Tullio CECCHERINI-SILBERSTEIN and Wolfgang WOESS}
\address{\parbox{.8\linewidth}{Dipartimento di Ingegneria, 
Universit\`a del Sannio,\\  
Corso Garibaldi 107, 82100 Benevento, Italy\\}}
\email{tceccher@mat.uniroma1.it}
\address{\parbox{.8\linewidth}{Institut f\"ur Mathematische Strukturtheorie,\\ 
Technische Universit\"at Graz,
Steyrergasse 30, 8010 Graz, Austria\\}}
\email{woess@TUGraz.at}
\date{October 31st, 2009} 
\thanks{The first author was partially supported by a visiting professorship
at TU Graz. The second author was partially supported by a visiting 
professorship at Universit\`a di Roma - La Sapienza and the
Austrian Science Fund project FWF-P19115-N18}
\subjclass[2000] {20F10, % Word problems, 
          %other decision problems, connections with logic and automata
                  68Q45,  %Formal languages and automata
                  05C25   %Graphs and groups.
		  }
\keywords{Finitely generated pair of groups, word problem, context-free
language, pushdown automaton, context-free graph}
\begin{abstract} Let $G$ be a finitely generated group, $A$ a finite set of
generators and $K$ a subgroup of $G$. We call the pair $(G,K)$ context-free
if the set of all words over $A$ that reduce in $G$ to an element of $K$
is a context-free language. When $K$ is trivial, $G$ itself is called
context-free; context-free groups have been classified more than 20 years ago
in celebrated work of Muller and Schupp as the virtually free groups.

Here, we derive some basic properties of such group pairs. Context-freeness
is independent of the choice of the generating set. It is preserved
under finite index modifications of $G$ and finite index enlargements of $K$.
If $G$ is virtually free and $K$ is finitely generated then $(G,K)$ is
context-free. A basic tool is the following: 
$(G,K)$ is context-free if and only if the Schreier graph of $(G,K)$ with
respect to $A$ is a context-free graph.  
\end{abstract}

\maketitle

\markboth{{\sf T. Ceccherini-Silberstein and W. Woess}}
{{\sf Context-free pairs of groups}}
\baselineskip 15pt

%%%%%%%%%%%%%%%%%%%%%%%%%%%%%%%%%%%
\section{Introduction}\label{sec:intro}
Let $G$ be a finitely generated group and $K$ a subgroup (not necessarily 
finitely generated). 
We can choose a finite set $A \subset G$ of generators  
such that every element of $G$ is of the form $g=g_1 \cdots g_n\,$, where
$n \ge 0$ and $g_1, \dots, g_n \in A$. Thus, $A$ generates $G$ as a semigroup.
We shall say that $(G,K)$ is \emph{context-free,} if -- loosely
spoken -- the language of all words over $A$ that represent an element
of $K$ is context-free. 

The precise definition needs some preparation.
Let $\Si$ be a finite \emph{alphabet} and $\psi: \Si \to G$
be a (not necessarily injective) mapping such that $A = \psi(\Si)$
satisfies the above finite generation property for $G$.
Then $\psi$ has a unique extension, also denoted $\psi$, as a monoid
homomorphism $\psi: \Si^* \to G$. 
Recall that $\Si^*$ consists of all \emph{words} $w=a_1 \cdots a_n$,
where $n \ge 0$ and $a_1, \dots, a_n \in \Si$ (repetitions allowed).
The number $n$ is the \emph{length} $|w|$ of $w$.
If $n = 0$ this means that $w = \epsilon$, the \emph{empty word}. 
This is the neutral element of $\Si^*$, and $\Si^*$ is a free monoid with the
binary operation of \emph{concatenation} of words. The extension of $\psi$ 
is of course given by
$$
\psi(a_1 \cdots a_n) = \psi(a_1) \cdots \psi(a_n)\,,
$$
where the product on the right hand side is taken in $G$.
Given these ingredients, we shall say that $\psi:\Si \to G$ is a
\emph{semigroup presentation} of $G$, referring to the fact that $A$ 
generates $G$ as a semigroup. 
A \emph{language} over $\Si$ is a non-empty subset of $\Si^*$.

\begin{dfn}\label{def:wordproblem}
The  \emph{word problem} of $(G,K)$ with respect to $\psi$ is the language
$$
L(G,K,\psi) = \{ w \in \Si^* : \psi(w) \in K \}\,.
$$
We say that the triple $(G,K,\psi)$ is \emph{context-free,} 
if $L(G,K,\psi)$ is a context-free language. 
\end{dfn}

A \emph{context-free grammar} is a quadruple $\Ccal = (\V,\Si,\Pb,S)$, 
where  $\V$ is a finite set of \emph{variables,} disjoint from 
the finite alphabet $\Si$ (the \emph{terminal symbols}), the variable  
$S$ is the \emph{start symbol,} and $\Pb \subset \V \times (\V \cup \Si)^*$ is a 
finite set of \emph{production rules.} We write $T \vdash u$ or 
$(T \vdash u) \in \Pb$ if $(T,u) \in \Pb$.  For $v, w \in (\V \cup \Si)^*$, 
we write $v \then w$ if $v = v_1Tv_2$ and $w=v_1uv_2$, where 
$u, v_1, v_2 \in (\V \cup \Si)^*$ and $T \vdash u$. This is a single derivation
step, and it is called \emph{rightmost,} if $v_2 \in \Si^*$. 
A \emph{derivation} is a sequence
 $v=w_0,w_1, \dots,w_k=w \in (\V \cup \Si)^*$ such 
that $w_{i-1} \then w_i\,$; we then write $v \thens w$. A  
\emph{rightmost derivation} is one where each step is rightmost. 
The succession of steps of any derivation $T \thens w \in \Si^*$ can be 
reordered so that it becomes a rightmost derivation.
For $T \in \V$, we consider the language  
$L_T = \{ w \in \Si^* : T \thens w \}$. The \emph{language generated by} 
$\Ccal$ is $L(\Ccal) =L_S$.  
  
A \emph{context-free language} is a language generated by a 
context-free grammar. 
As a basic reference for Language and Automata Theory, we refer to
the magnificent monograph of {\sc Harrison~\cite{Ha}}.

The above definition of a context-free pair, or rather triple, $(G,K,\psi)$
makes sense when $G$ is a finitely generated monoid and $K$ is a sub-monoid, 
but here  we are interested in groups. 
%The most typical situation arises when $G$ is finitely generated, and $H=G$. 
When in addition $K = \{ 1_G \}$, this leads 
to the notion of $G$ being a \emph{context-free group.} In two celebrated papers, 
{\sc Muller and Schupp~\cite{MS1}, \cite{MS2}} have carried out a detailed 
study of context-free groups and more generally, context-free \emph{graphs}. In
particular, context-freeness of a group is independent of the particular 
choice of the generating set $A$ of $G$. The main result of \cite{MS1},
in combination with a fundamental theorem of {\sc Dunwoody}~\cite{Du2},
is that a finitely generated group is context-free if and only if it is
virtually free, that is, it contains a free subgroup with finite index.
(In \cite{MS1}, it is assumed that $A=A^{-1}$ and that 
$\psi:\Si \to A=\psi(\Si)$ 
is  one-to-one, but the results carry over immediately to the more general 
setting where those two properties are not required.) 

Previously, {\sc Anisimov}~\cite{An} had shown that the groups whose word 
problem
$L(G,\{1_G\},\psi)$ is \emph{regular} (see \S \ref{sec:Schreier} for the
definition) are precisely the finite groups.

The abovementioned context-free graphs are labelled, rooted graphs with finitely many isomorphism
classes of \emph{cones}. The latter are the connected components of the graph 
that remain after removing a ball around the root with arbitrary radius. See 
\S \ref{sec:cfgraphs} for more precise details. As shown in \cite{MS2}, there is
a natural correspondence between such graphs and \emph{pushdown automata},
which are another tool for generating context-free languages; see 
\S \ref{sec:pushdown}.

Among subsequent work, we mention {\sc P\'elecq}~\cite{Pe} and 
{\sc S\'enizergues}~\cite{Sen}, who studied actions on, resp. quotients of 
context-free graphs. Group-related examples occur also in 
{\sc Ceccherini-Silberstein and Woess}~\cite{CeWo1}.

More recently, {\sc Holt,  Rees,  R\"over and Thomas}~\cite{Ho} have introduced and studied
\emph{co-context-free} groups, which are such that the complement of 
$L(G,\{1_G\},\psi)$ is context-free, see also {\sc Lehnert and Schweitzer}~\cite{LeSc}.
This concept has an obvious extension to co-context-free pairs of groups,
resp. graphs, on whose examination we do not (yet) embark. 

In the present notes, we collect properties and examples of context-free pairs
of groups $(G,K)$.

\begin{itemize}
\item The language $L(G,K,\psi)$ is regular if and only if the index $[G:K]$
of $K$ in $G$ is finite  (Proposition \ref{pro:regular}).
\item The property that $L(G,K,\psi)$ is context-free does not depend
on the specific choice of the semigroup presentation $\psi$, so that
context-freeness is just a property of the pair $(G,K)$, a consequence of
Lemma \ref{lem:translator}.
\item If $(G,K)$ is context-free then $L(G,K,\psi)$ is a deterministic
context-free language (see \S \ref{sec:pushdown} for the definition)
for any semigroup presentation $\psi: \Si \to G$ (Corollary \ref{cor:pairs}.a).
\item If $(G,K)$ is context-free and $H$ is a finitely generated subgroup 
of $G$, then the pair $(H,K\cap H)$ is context-free (Lemma
\ref{lem:translator}).
\item If $[G:H] < \infty$ then $(G,K)$ is 
context-free if and only if $(H,K\cap H)$ is context-free (Proposition
\ref{pro:extend} \& Lemma \ref{lem:normal}).
\item If $(G,K)$ is context-free and $H$ is a subgroup of $G$ with $K \le H$ 
and $[H:K] < \infty$ then $(G,H)$ is context-free (Lemma \ref{lem:normal}).
\item If $K$ is finite then $G$ is context-free if and only if $(G,K)$
is context-free (Lemma \ref{lem:Kfinite}).
\item If $(G,K)$ is context-free then $(G,g^{-1}Kg)$ is context-free for every
$g \in G$ (Corollary \ref{cor:pairs}.b). 
\item If $G$ is virtually free and $K$ is a finitely generated subgroup of 
$G$ then $(G,K)$ is context-free (Corollary \ref{cor:fingen}).
\end{itemize}
Several of these properties rely on the following.
\begin{itemize}
\item A fully deterministic, symmetric labelled graph (see 
\S \ref{sec:Schreier} for definitions) is context-free in the sense of
Muller and Schupp if and only if the language of all words which are labels
of a path that starts and ends at a given root vertex is context-free 
(Theorems \ref{thm:characterize1} and \ref{thm:characterize2}).
\end{itemize}

The (harder) ``if'' part is not contained in
previous work. It implies the following.

\begin{itemize}
\item The pair $(G,K)$ is context-free if and only if for some ($\!\!\iff\!$
any) symmetric semigroup presentation 
$\psi: \Si\to G$, the \emph{Schreier graph} of $(G,K)$ with respect to 
$\psi$ is a context-free graph. (See again \S \ref{sec:Schreier} for 
precise definitions). 
\end{itemize}

In a second paper \cite{Wcf2}, a slightly more general approach to 
context-freeness of graphs via cuts and tree-sets is given. It allows to 
show that certain
structural properties (``irreducibility'') are preserved under 
finite-index-modifications of the underlying pair of groups. This is then
applied to random walks, leading in particular to results on the asymptotic
behaviour of transition probabilities.

In concluding the Introduction, we remark that with the exception of some 
``elementary'' cases, context-free pairs of groups are always pairs with 
more than one \emph{end}. Ends of pairs of groups were studied, e.g., by 
{\sc Scott}~\cite{Sc}, {\sc Swarup}~\cite{Sw} and {\sc Sageev}~\cite{Sa}. 
In particular, the interplay between context-freeness of pairs and
decomposition as amalgamated products or HNN-extensions needs still
to be explored.  

\smallskip

\noindent{\bf Acknowledgement.} We are grateful to Wilfried Imrich 
and R\"ognvaldur G. M\"oller
%and V. I. Trofimov
for useful hints and discussions.

%%%%%%%%%%%%%%%%%%%%%%%%%%%%%%%%%%%%%%%%%%%%%%%%%%%%%%%%%%%%%%%%%%%%

\section{Schreier graphs,  and the regular 
case}\label{sec:Schreier}

Let $\Si$ be a finite alphabet. 
A \emph{directed graph labelled by $\Si$} is a triple  
$(\VV, \EE, \ell)$, where $\VV$ is the (finite or countable) set of \emph{vertices,} 
$\EE \subset \VV \times \Si \times \VV$ is the set of \emph{oriented, labelled 
edges} and $\ell :\EE \ni (x,a,y) \mapsto a \in \Si$ is the \emph{labelling}
map. 

For an edge $e =(x,a,y) \in \EE$, its \emph{initial vertex} is $e^- = x$ and 
its \emph{terminal vertex} is $e^+ = y$, and we say that $e$ is \emph{outgoing}
from $x$ and \emph{ingoing} into $y$. If $y=x$ then $e$ is a \emph{loop},
which is considered both as an outgoing and as an ingoing edge. 
We allow \emph{multiple edges}, i.e., edges of the form 
$e_1 = (x, a_1, y)$ and $e_2 = (x, a_2, y)$ with 
$a_1 \neq a_2$, but here we exclude multiple edges where also the
labels coincide. The graph is always assumed to be \emph{locally finite}, that 
is, every vertex is an initial or terminal vertex of only finitely many edges.
We also choose a fixed vertex $o \in \VV$, the 
\emph{root} or \emph{origin.} We shall often just speak of the graph $\VV$,
keeping in mind the presence of $\EE$ and $\ell$.

We call $\XX$ \emph{fully labelled} if at every vertex, each 
$a \in \Sigma$ occurs as the label of at least one outgoing edge.
We say that $\XX$ is \emph{deterministic} if at every vertex all outgoing edges 
have distinct labels, and \emph{fully deterministic} if
it is fully labelled and deterministic.
%for every vertex $x$ and every label $a \in \Si$, there is precisely one ingoing
%and one outgoing edge at $x$ with label $a$. (The two may coincide in case
%of a loop.) 
Finally, we say that $\XX$ is \emph{symmetric} or \emph{undirected}
if there is a proper involution $a \mapsto a^{-1}$ of $\Si$ (i.e., 
$(a^{-1})^{-1} = a$, excluding the possibility that 
$a^{-1}=a$) such that for each edge $e=(x,a,y) \in \EE$, also 
the \emph{reversed edge} $e^{-1}=(y,a^{-1},x)$ belongs to $\EE$. 

%The sum of the number of edges outgoing from and the number of edges ingoing into a given vertex $x \in \VV$ is called the {\it degree} of $x$ and it is denoted by $\bd(x)$. If there exists $d \in {\mathsf N}$ such that
%$\bd(x) = d$ for all $x \in \VV$ one says that $X$ is {\it regular} of degree $d$.

A \emph{path} in $\XX$ is a sequence $\pi = e_1 e_2 \dots e_n$ of edges such 
that $e_i^+ = e_{i+1}^-$ for $i=1,\ldots, n-1$. The vertices $\pi^- = e_1^-$ 
and $\pi^+ = e_n^+$ are the \emph{initial} and the \emph{terminal} vertex of
$\pi$. The number  $|\pi| =n$ is the \emph{length} of the path.
The \emph{label} of $\pi$ is 
$\ell(\pi) = \ell(e_1) \ell(e_2) \cdots \ell(e_n) \in \Si^*$. We also admit 
the \emph{empty path} starting and ending at a vertex $x$, whose label is
$\epsilon$. 
Denote by %${\mathcal P}(X)$ the set of all paths in $X$ and, for $x,y \in \VV$,  
$\Pi_{x,y}=\Pi_{x,y}(\XX)$ the set of all paths $\pi$ in $\XX$ with initial 
vertex $\pi^- = x$ and terminal vertex $\pi^+ = y$. The following needs no
proof.

\begin{dlm}\label{def:pathfromx}
Let $(\VV, \EE, \ell)$ be a labelled graph, $x \in \VV$ and $w \in \Si^*$.
We define $\Pi_x(w) = \{ \pi : \pi^-=x\,,\; \ell(\pi) = w\}$, the set
of all paths that start at $x$ and have label $w$. The set of all terminal
vertices of those paths is denoted $x^w = \{ \pi^+ : \pi \in \Pi_x(w)\}$.

Analogously, we define  $\overline \Pi_x(w) = 
\{ \pi : \pi^+=x\,,\; \ell(\pi) = w\}$, the set
of all paths that terminate at $x$ and have label $w$, and write 
$x^{-w} = \{ \pi^- : \pi \in \overline \Pi_x(w)\}$.

If $\XX$ is fully labelled, then $\Pi_x(w)$ is always non-empty.

If $\XX$ is deterministic, then $\Pi_x(w)$ has at most one element, and
if that element exists, it is denoted $\pi_x(w)$, while $x^w$ just denotes
its endpoint. 

If $\XX$ is fully deterministic, then $x^w$ is a unique
vertex of $\XX$ for every $x\in \VV$, $w \in \Si^*$.

Finally, if $\XX$ is symmetric (not necessarily deterministic), 
then $\overline \Pi_x(w) = \Pi_x(w^{-1})$, where for $w=a_1 \cdots a_n$,
one defines $w^{-1}= a_n^{-1} \cdots a_1^{-1}$.
\end{dlm}

With a labelled, directed graph as above, 
we can associate various languages. We can, e.g., consider the language
\begin{equation}\label{eq:Lxy}
L_{x,y} = L_{x,y}(\XX) = \{ \ell(\pi) : \pi \in \Pi_{x,y}(\XX)\}\,,
\quad\text{where}\;\ x,y \in \VV\,.
\end{equation}

\begin{dfn}\label{def:schreier}
Let $G$ be a finitely generated group, $K$ a subgroup and 
$\psi : \Si \to G$ a semigroup presentation of $G$. %a map from a finite alphabet $\Si$ 
The \emph{Schreier graph} $\XX = \XX(G,K,\psi)$ has vertex set
$$
\VV= K \backslash G = \{ Kg: g \in G \}
$$ 
(the set of all right $K$-cosets in $G$), and the set of labelled, directed
edges 
$$
\EE = \{e = (x,a,y): x = Kg\,,\; y = Kg\psi(a)\,,\;\text{where}\; 
g \in G\,,\; a \in \Si\}\,.
$$
\end{dfn}

$\XX$ is a rooted graph with origin $o = K$, the right coset corresponding 
to the neutral element $1_G$ of the group $G$. The Schreier graph is fully 
deterministic. It is also \emph{strongly connected:} for every pair 
$x, y \in \VV$, there is a path from $x$ to $y$. (This follows from the
fact that $\psi(\Si)$ generates $G$ as a semigroup.) When $K = \{1_G\}$ then
we write $X(G,\psi)$. This is the \emph{Cayley graph} of $G$ with respect
to $\psi$, or more loosely speaking, with respect to the set $\psi(\Si)$
of generators.

Note that $\XX$ can have the loop $e = (x,a,x) \in \EE$ with $x = Kg$. This
holds if and only if $\psi(a) \in g^{-1}Kg$. It can also have the multiple edges 
$e_1 = (x, a_1, y)$ and $e_2 = (x, a_2, y)$ with $x = Kg$ and $a_1 \neq a_2$.
This occurs 
if and only if  $\psi(a_2)\psi(a_1)^{-1} \in g^{-1}Kg$. In particular, there 
might be multiple loops. The following is obvious.

\begin{lem}\label{lem:cfpaths} 
Let $K$ be a subgroup of $G$ and $\psi : \Si \to G$ be a semigroup presentation of $G$. 
Then
$$
L(G,K,\psi) = L_{o,o}(\XX)
$$
is the language of all labels of closed paths starting and ending at $o=K$
in the Schreier graph $\XX(G,K,\psi)$. 
\end{lem}

%%%%%%%%%%%%%%%%%%%%%%%%%%%%%%%%%%%%%%%%%%%%%%%%%%%%%%%%%%%%%%%%%%}

A context-free grammar  $\Ccal = (\V,\Si,\Pb,S)$ and the language $L(\Ccal)$
are called \emph{linear,} if every production rule in $\Pb$ is of the form 
$T \vdash v_1Uv_2$ or $T \vdash v$, where $v,v_1,v_2 \in \Si^*$ and 
$T, U \in \V$. If furthermore in this
situation one always has $v_2 = \ep$ (the empty word), then grammar and
language are called \emph{right linear} or \emph{regular.}
% Analogously, it is called
%{\it left linear,\/} if instead one always has $v_1 = \ep$. In both cases,
%language and grammar are also called {\it regular.\/} (It is well known
%that left and right linear languages are the same, i.e., every left linear
%language is also generated by a right linear grammar and conversely.)

A \emph{finite automaton} $\Acal$ consits of a finite directed graph
$\XX = (\VV, \EE, \ell)$ with label set $\Si$ and labelling map $\ell$,
together with a root vertex $o$ and a nonempty set $\FF \subset \VV$.
The vertices of $\XX$ are called the \emph{states} of $\Acal$, the root $o$ 
is the \emph{initial} state, and the elements of $\FF$ are the \emph{final} 
states.
The automaton is called \emph{(fully) deterministic} provided the labelled 
graph $\XX$ is (fully) deterministic. 
The language \emph{accepted} by $\Acal$ is
$$
L(\Acal) = \bigcup_{x \in \FF} L_{o,x}(\XX)\,.
$$
If $\Acal$ is deterministic, then for each $w \in L(\Acal)$ there is a 
unique path $\pi \in \bigcup_{x \in \FF} \pi_{o,x}(\XX)$ such
that $\ell(\pi) = w$. A state $y \in \VV$ is called \emph{useful} if there 
is some word $w \in L$ such that the vertex $y$ lies on a path 
in $\bigcup_{x \in \FF} \pi_{o,x}(\XX)$ with label $w$. 
It is clear that we can remove all useless states and their ingoing 
and outgoing edges to obtain an automaton which accepts the same 
language and is \emph{reduced:} it has only useful states. 
% 
%Let $X = (\VV, \EE, {\mathcal L})$ be a $\Si$-{labelled oriented graph}. Fix a vertex ${\bf o} \in \VV$, called
%the {\it initial} state, and a non-empty subset ${\mathcal F}
%\subset \VV$, the set of {\it final} states.  The triple
%${\mathcal A} = (X, {\bf o}, {\mathcal F})$ is a {\it one-initial-state automaton}.
%The automaton is called {\it (strongly) deterministic} provided the labelled graph $X$ si (strongly) deterministic. h
%A path $\pi \in {\mathcal P}_{{\bf o},f}$ with initial vertex ${\bf o}$ and 
%terminal vertex a final state $f \in F$ is called {\it admissible}.
%
%
%The language associated with an automaton ${\mathcal A} = (X, {\bf o}, {\mathcal F})$ is the set  
%$$L({\mathcal A}) = \{{\mathcal L}(\pi): \pi \in {\mathcal P}_{{\bf o},f}, f \in F\}$$ 
%of the labels of all admissible paths of $X$. 
%
%One says that ${\mathcal A}$ is {\it unambiguous} if
%for any word $w \in L({\mathcal A})$ there is a unique path $\pi \in {\mathcal P}(X)$ such that  $w = {\mathcal L}(\pi)$. Clearly ${\mathcal A}$ is unambiguous if and only if it is {deterministic}.
%
%
%For a strongly deterministic one-initial-state  automaton ${\mathcal A}$ 
%consider the  map $t \colon \VV \times \Si \to \VV$ defined by $t(p,a) = q$ 
%where $q$ is the terminal vertex of the unique
%edge $e = (p,a,q)$ outgoing form $p$ with label $a$. It natuarally extends to a map $t \colon \VV \times \Si^* \to \VV$, called the {\it state transition map}, by setting 
%$t(p,aw) = t(t(p,a),w)$ for all $p \in \VV$, $a \in \Si$ and $w \in \Si^*$.
%
It is well known \cite[Chapter 2]{Ha} that a language $L \subseteq \Si^*$ 
is regular if and only if $L$ is accepted by
some deterministic finite automaton.

The following generalizes Anisimov's \cite{An} characteriziation of groups
with regular word problem, and also simplifies its proof, as well as the 
simpler one of \cite[Lemma 1]{MS1}.

\begin{pro}\label{pro:regular} 
Let $G$ be a finitely generated group, $K$ a subgroup and 
$\psi : \Sigma \to G$ a semigroup presentation of $G$. 
Then $(G,K)$ has regular word problem with respect to $\psi$ if and only if 
$K$ has finite index in $G$. %, namely $[G : H] < \infty$.
\end{pro}

\begin{proof} Suppose first that the index of $K$ in $G$ is finite. 
Consider the finite automaton 
$\Acal = (\XX, o,\{o\})$ where $\XX$ is the Schreier graph $\XX(G,K,\psi)$,
and the initial and unique final state is $o=K$ (as a vertex of $\XX$). 
Then $L(G,K,\psi) =L(\Acal)$: indeed, $w \in \Si^*$ belongs to 
$L(G,K,\psi)$, i.e. $\psi(w) \in K$, if and only if $K = K\psi(w)$.  
This shows that $L(G,K,\psi)$ is regular.

Conversely, suppose that $L=L(G,K,\psi)$ is regular and accepted by
the reduced, deterministic finite automaton $\Acal = (\XX, o, \FF)$. 
For $y \in \VV$ there is some word $w \in L$
such that the vertex $y$ lies on the unique path from $o$ to $\FF$ 
with label $w$. 
We choose one such $w$ and let $w_y$ be the label of the final piece 
of the path, starting at $y$ and ending at $\FF$. We set 
$g_y = \psi(w_y)^{-1} \in G$.
   
Let $g \in G$. There are $w, \overline{w} \in \Si^*$ with $\psi(w)=g$
and $\psi(\overline{w}) =g^{-1}$. Thus, $w\overline{w} \in L=L(G,K,\psi)$, 
and there is
a (unique) path $\pi$ with label $w\overline{w}$ from $o$ to some final state.
Now consider the initial piece $\pi_w$ of $\pi$, that is, the path starting at 
$o$ whose label is our $w$ that we started with. 
[Thus, we have proved that such a path $\pi_w$ must exist in $\XX$~!] 
Let $y$ be the
final state (vertex) of $\pi_w$. Then clearly $ww_y \in L(\Acal)$, which 
means that $g g_y^{-1} = \psi(w w_y) \in K$. Since $\psi(\Si^*) = G$, it
follows that
$$
G = \bigcup_{y \in \VV} Kg_y\,,
$$
and $K$ has finitely many cosets in $G$.

\end{proof}

\begin{cor}\label{cor:regular} Let $G$ be finitely generated and $K$ a subgroup.
Then the property of the pair $(G,K)$ to
have a regular word problem is independent of the semigroup presentation of $G$.
\end{cor}

We shall see that the same also holds in the context-free case.
Another corollary that we see from the proof of Proposition \ref{pro:regular} 
is the following.

\begin{cor}\label{cor:automaton} Let $G$ be finitely generated and
$K$ a subgroup with finite index. Then for any semigroup presentation
$\psi: \Si \to G$, any reduced deterministic
automaton $\Acal= (\XX, o, \FF)$ that accepts $L(G,K,\psi)$ has a
surjective homomorphism (as a labelled oriented graph with root $o$) onto the
Schreier graph $\XX(G,K,\psi)$. 
Also, the labelled graph $\XX$ is fully deterministic.

%at each vertex 
%of $\XX$ there is an outgoing edge with label $a$ for each $a \in \Si$.
\end{cor}

\begin{proof} Let $\Acal = (\XX, o, \FF)$ be deterministic and
reduced, as in part 2 of the proof of Proposition \ref{pro:regular}.

Let $y \in \VV$, and recall the construction of the label $w_y$ of a
path from $y$ to $\FF$, and $g_y = \psi(w_y)^{-1} \in G$.
If $v$ is another path from $y$ to $\FF$, and $h = \psi(v)^{-1}$,
then we can take $w \in L_{o,y}$ (which we know to be non-empty) and
find that $ww_y, wv \in L(G,K,\psi)$, so that $\psi(w) \in Kg_y \cap Kh$.
Thus $Kg_y = K\psi(w) = Kh$, and the map $\kappa : \VV \to K \backslash G,\;$ 
$y \mapsto Kg_y$ is well defined. 
It has the property that when $w \in L_{o,y}$, then $K\psi(w)=Kg_y$.
The map $\kappa$ is clearly surjective, and $\kappa(o)=K$ by construction.

Now let $y \in \VV$ and $a \in \Si$. Take $w \in L_{o,y}$ and consider the
word $wa$. Again by part 2 of the proof of Proposition \ref{pro:regular}, 
there is a unique path 
$\pi_{wa}$ in $\XX$ starting at $o$ with label $wa$. If $y$ is its final
vertex, then there is the edge $e = (y,a,z)$ in $\XX$. 
In this situation, $\kappa(z) = K\psi(wa)= Kg_y\psi(a) = \kappa(y)\psi(a)$.
This means that in the Schreier graph, there is the edge with label $a$ from
$\kappa(y)$ to $\kappa(z)$. Therefore $\kappa$ is a homomorphism of labelled graphs.
\end{proof}

The following simple example shows that, in general, the map $\kappa$ 
constructed in the proof of the previous corollary is not injective. 

Let $G = \mathbb Z_2$ = \{1,t\} be the group of order two and $K = \{1\}$ the 
trivial subgroup.
Let $\Si =\{a\}$ and consider the presentation 
$\psi \colon \Si \to G$ such that $\psi(a) = t$.
Then $L(G,K,\psi) = \{a^{2n}: n \geq 0 \}$. 

\begin{center}
\begin{picture}(380,160)
\put(20,70){\circle{28}}
\put(80,70){\circle{28}}
\put(120,70){\circle{28}}
\put(180,70){\circle{28}}
\put(220,70){\circle{28}}
\put(280,110){\circle{28}}
\put(280,30){\circle{28}}
\put(340,70){\circle{28}}
\put(16,67){$1$}
\put(78,68){$t$}
\put(116,68){$o$}
\put(216,68){$o$}
\put(336,68){$f$}
\qbezier[100](30,80)(50,90)(70,80) 
\qbezier[100](30,60)(50,50)(70,60)
\qbezier[100](130,80)(150,90)(170,80) 
\qbezier[100](130,60)(150,50)(170,60)

\put(52,85){\vector(1,0){1}}
\put(48,55){\vector(-1,0){1}}
\put(152,85){\vector(1,0){1}}
\put(148,55){\vector(-1,0){1}}

\put(47,90){$a$}
\put(47,45){$a$}
\put(147,90){$a$}
\put(147,45){$a$}

\put(312,94){$a$}
\put(312,42){$a$}
\put(242,94){$a$}
\put(242,42){$a$}
\put(232,78){\line(4,3){35}}
\put(293,36){\line(4,3){35}}
\put(232,78){\vector(4,3){20}}
\put(293,36){\vector(4,3){20}}

\put(231,62){\line(4,-3){36}}
\put(293,105){\line(4,-3){35}}
\put(231,62){\vector(4,-3){20}}
\put(293,105){\vector(4,-3){20}}
\end{picture}
\end{center}
\begin{center}\vspace{-.5cm}
\emph{Figure 1.}
\end{center}
\vspace{.5cm}
In Figure 1 above we have represented, in order, the Schreier graph 
$X(G, K, \psi)$ (which is nothing but the Cayley graph of $G$ w.r. to $\psi$), 
and two automata $\Acal_1$ and $\Acal_2$. As usual $o$ denotes the
origin, while the sets of final states  are $F_1 =\{o\}$ and 
$F_2 = \{o,f\}$, respectively. We have
$L(\Acal_1) = L(\Acal_2) = L(G,K,\psi)$.

%%%%%%%%%%%%%%%%%%%%%%%%%%%%%%%%%%%%%%%%%%%%%%%%%%%%%%%%%%%%%%%%%%
\section{Pushdown automata}\label{sec:pushdown}

Besides grammars, we shall need another instrument for generating context-free
languages. A \emph{pushdown automaton} is a $7$-tuple 
$\Acal = (\Q,\Si,\Z, \de, q_0, \Q_f, z_0)$, where $\Q$ is a finite set of \emph{states,}
$\Si$ the input alphabet as above, $\Z$ a finite set of \emph{stack symbols,}
$q_0 \in Q$ the initial state,
$\Q_f \subset \Q$ the set of final states, and $z_0 \in Z$ is the \emph{start
symbol.} Finally, the function
$\de: \Q \times (\Si \cup \{ \epsilon \}) \times (\Z\cup \{ \epsilon \}) 
\to \Pcal_{\text{\rm fin}}(\Q \times \Z^*)$
is the \emph{transition function.} Here, 
$\Pcal_{\text{\rm fin}}(\Q \times \Z^*)$ stands for the
collection of all finite subsets of $\Q \times \Z^*$.

The autmaton works in the following way. At any time, it is in some state 
$p \in \Q$, and the stack contains a word $\zeta \in \Z^*$. The automaton 
reads a word $w \in \Si^*$ from the ``input tape'' letter by letter from 
left to right.
If the current letter of $w$ is $a$, the state is $p$ and the top (=rightmost) 
symbol of the stack word $\zeta$ is $z$, it performs one of the following
transitions.    

\smallskip

(i) $\;\Acal$ selects some $(q,\zeta') \in \de(p,a,z)$, changes into state $q$, 
moves to the next position on the input tape (it may 
be empty if $a$ was the last letter of $w$), and replaces the rightmost 
symbol $z$ of $\zeta$ by $\zeta'$, or

(ii) $\;\Acal$ selects some $(q,\zeta') \in \de(p,\epsilon,z)$, changes into state $q$, 
remains at the current position on the input tape (so that
$a$ has to be treated later), and replaces the rightmost symbol $z$ of $\zeta$
by $\zeta'$. 

If both $\de(p,a,z)$ and $\de(p,\epsilon,z)$ are empty then $\Acal$ halts.

\smallskip

The automaton is also allowed to continue to work when the stack is 
empty, i.e., when $\zeta = \epsilon$. Then the automaton acts
in the same way, by putting $\zeta'$ in the stack when it has selected
$(q,\zeta') \in \de(p,a,\epsilon)$ in case (i), resp. 
$(q,\zeta') \in \de(p,\epsilon,\epsilon)$ in case (ii). 

We say that $\Acal$ \emph{accepts} a word $w \in \Si^*$ if starting at the
state $q_0$ with only $z_0$ in the stack and with $w$ on the input tape,
after finitely many transitions the automaton can reach a final state with 
empty stack and empty input tape. The language accepted by $\Acal$ is denoted
$L(\Acal)$.
%(Since the automaton
%is allowed to continue to work when the stack is empty, we might
%as well omit the start symbol and allow $\Acal$ to start with empty stack.)

The pushdown automaton is called \emph{deterministic} if for any 
$p \in \Q$, $a \in \Si$ and  $z \in \Z \cup \{\epsilon\}$, it has at most 
one option what to do next, that is,
$$
|\de(p,a,z)| + |\de(p,\epsilon,z)| \le 1\,.
$$
(Here, $|\,\cdot\,|$ denotes cardinality.)

It is well known \cite{Ha} that a language is context-free if and
only if it is accepted by some pushdown automaton. A context-free language is
called \emph{deterministic} if it is accepted by a deterministic pushdown
automaton. We also remark here 
that a deterministic context-free language $L$  is \emph{un-ambiguous}, which 
means that it is generated by some context-free grammar in which every word of 
$L$ has precisely one rightmost derivation.

The following lemma is modelled after the indications of
\cite[Lemma 2]{MS1}. For the sake of completeness, we include the full proof.

\begin{lem}\label{lem:translator}
Suppose that $G, K, \Si$ and $\psi:\Si \to G$ are as above.
Let $H$ be a finitely generated subgroup of $G$, and let $\Si'$ be another 
alphabet and $\psi':\Si' \to H$ be such 
that $F' = \psi'(\Si')$ generates $H$ as a semigroup.

Then, if $L(G,K,\psi)$ is context-free, also $L(H,K\cap H,\psi')$ is 
context-free, and if in addition $L(G,K,\psi)$ is deterministic, then so is 
$L(H,K\cap H,\psi')$.
\end{lem}

\begin{proof} We start with a pushdown automaton 
$\Acal = (\Q,\Si,\Z, \de, q_0, \Q_f, z_0)$ that accepts $L(G,K,\psi)$.

For each $b \in \Si'$, there is $u(b) \in \Si^*$ such that 
$\psi'(b)=\psi\bigl(u(b)\bigr)$, and we may 
choose $u(b)$ to have length $\ge 1$. 
Thus, 
$$
w' = b_1 \cdots b_n \in L(H,K\cap H,\psi') \iff u(b_1) \cdots u(b_n) \in
L(G,K,\psi)\,.
$$
With this in mind, we modify $\Acal$ in order to obtain a pushdown
automaton $\Acal'$ that accepts $L(H,K\cap H,\psi')$. Our $\Acal'$ has to 
translate any $w' = b_1 \cdots b_n \in (\Si')^*$ into 
$w=u(b_1) \cdots u(b_n) \in \Si^*$
and to use $\Acal$ in order to check whether $w \in L(G,K,\psi)$.

Let $m+1 = \max \{ |u(b)| : b \in \Si' \}$. If $m=0$ then the only modification
of $\Acal$ needed is to replace $\Si$ by its subset $\Si'$ and to use the
resulting restriction of the transition function.

Otherwise, we set $\Si_m = \Si \cup \Si^2 \cup \cdots \cup \Si^m$.
For $v \in \Si^+ = \Si^* \setminus \{ \epsilon\}$, we denote by
$v_+$ its subword obtained by deleting the first letter.
We define  $\Q' = \Q \cup (\Q \times \Si_m)$ and
$\Acal' = (\Q',\Si',\Z, \de', q_0, \Q_f, z_0)$ with the
transition function $\de'$ as follows. For each $p \in \Q$ and $z \in \Z$,
$$
\begin{aligned}
\de'(p,\epsilon,z) &= \de(p,\epsilon,z) \,,\\ 
\de'(p,b,z) &= \de\bigl(p,a,z\bigr) \,,\quad\text{if}\; u(b) = a \in \Si\,,\\
\de'(p,b,z) &= 
\Bigl\{ \bigl((q,u(b)_+), \zeta\bigr) : (q,\zeta) \in \de(p,a,z)\Bigr\} \,,
\qquad\quad\text{if}\; u(b) \in  a\Si^+\,,\\
\de'\bigl( (p,v),\epsilon,z\bigr) &=
\Bigl\{ \bigl((q,v), \zeta\bigr) : 
    (q,\zeta) \in \de(p,\epsilon,z)\Bigr\}\\ 
&\qquad\cup\; \Bigl\{ \bigl((q,v_+), \zeta\bigr) : 
    (q,\zeta) \in \de(p,a,z)\Bigr\}\,,
\qquad\quad\;\;\text{if}\; v \in a\Si^+\,,\\
\de'\bigl( (p,a),\epsilon,z\bigr) &=
\Bigl\{ \bigl((q,a), \zeta\bigr) : (q,\zeta) \in \de(p,\epsilon,z)\Bigr\} 
\;\cup\; \de(p,a,z)\,,
\qquad\text{if}\; a\in \Si\,.
\end{aligned}
$$
Thus, the new states of the form $(p,v)$ with $1 \le |v| < m$
serve to remember the terminal parts $v$ of the words $u(b)$, $b \in \Si'$.
This automaton accepts $L(G,K,\psi')$, and it is deterministic,
if $\Acal$ has this property.
\end{proof}

\begin{cor}\label{cor:indep} Being context-free is a property of the
pair $(G,K)$ that does not depend on the specific choice of 
the alphabet $\Si$ and the map $\psi: \Si \to G$ for which $\psi(\Si)$
generates $G$ as a semigroup. 
\end{cor}

Therefore, it is justified to refer to
the contxt-free pair $(G,K)$ rather than to the triple $(G,K,\psi)$.
Furthermore, whenever this is useful, we may restrict
attention to the case when the graph $\XX(G,K,\psi)$ is symmetric: we say that
\emph{$\psi$ is symmetric,} if there is a proper involution 
$a \mapsto a^{-1}$ of $\Si$ such that $\psi(a^{-1}) = \psi(a)^{-1}$
in $G$. (Again, it is not necessary to assume that $\psi$ is one-to-one,
so that we have that $a^{-1} \ne a$ even when $\psi(a)^2 = 1_G$.)

\begin{pro}\label{pro:extend} Let $G$ be finitely generated, $H$ be 
a subgroup with $[G:H] < \infty$. If $K$ is a subgroup of $H$ then $(G,K)$ is 
context-free if and only if $(H,K)$ is context-free.
\end{pro} 

\begin{proof} The ``only if'' is contained in Lemma \ref{lem:translator}.
(Observe that $H$ inherits finite generation from $G$, since $[G:H] < \infty$.)

For the converse, we assume that $(H,K)$ is context-free and 
let $\psi: \Si \to H$ and $\psi': \Si' \to G$ be semigroup presentations 
of $H$ and $G$, respectively. There is a pushdown automaton 
$\Acal = (\Q,\Si,\Z, \de, q_0, \Q_f, z_0)$ that accepts $L(H,K,\psi)$.

Let $F$ be a set of representatives of the right cosets of $H$ in $G$,
with $1_G \in F$. Thus, $|F| < \infty$, and 
$$
G = \biguplus_{g \in F} Hg\,,
$$
For every $g \in F$ and $b\in \Si'$ there is a unique 
$\bar g = \bar g (g,b) \in F$ such that $g\psi'(b) \in H\bar g$.
Therefore there is a word $u = u(g,b) \in \Si^*$ such that 
$$
g\psi'(b) = \psi\bigl(u(g,b)\bigr) \bar g (g,b)\,.
$$
An input word $w=b_1 \cdots b_n$ is transformed recursively into 
$u_1 \cdots u_n\,$, along with the sequence $g_0, g_1, \dots,g_n$
of elements of $F$ that indicate the current $H$-coset at each step:
$$
g_0 = 1_G\,;\quad u_k = u(g_{k-1},b_k)\AND g_k = \bar g(g_{k-1},b_k)\,.
$$
Then $\psi'(w) \in K$ if and only if $g_n = 1_G$ and
$\psi(u_1 \cdots u_n) \in K$.  

 Thus, our new automaton $\Acal'$ recalls at each step the current coset 
$Hg_{k-1}\,$, which is multiplied on the right by $\psi(b_k)$, where $b_k$ is 
the next input letter.
Then the new coset is $H\bar g(g_{k-1},b_k)$, and  $\Acal'$ simulates what 
$\Acal$ does next upon reading $u(g_{k-1},b_k)$. Then 
$w$ is accepted when at the end the coset is $H = H1_G$ and $\Acal$ is in a 
final state.

The simple task to write down this automaton in detail is left to the reader.
\end{proof}

%%%%%%%%%%%%%%%%%%%%%%%%%%%%%%%%%%%%%%%%%%%%%%%%%%%%%%%%%%%%%%%%%%%%%%%%%%%%%%%%%%%%%%%%%%%%%%%%%%%%%%%%%%
\section{Context-free graphs}\label{sec:cfgraphs}
In this section, we assume that $(\VV,E,\ell)$ is symmetric. 
We may think of each pair of oppositely oriented edges
$(x,a,y)$ and $(y,a^{-1},x)$ as one non-oriented edge, so that
$\XX$ becomes an ordinary graph with symmetric neighbourhood relation,
but possibly multiple edges and loops.
If it is in addition fully deterministic, then $\XX$ is a 
\emph{regular graph}, that is, the number of outgoing edges
(which coincides with the number of ingoing edges) at each vertex is $|\Si|$. 
Attention: if we consider non-oriented edges, then each loop at $x$ has to
be counted twice, since it corresponds to two oriented edges of the
form $(x,a,x)$ and $(x,a^{-1},x)$.
For all our purposes it is natural to require that $\XX$ is \emph{connected:}
for any pair of vertices $x,y$ there is a path from $x$ to $y$.
The distance $d(x,y)$ is the minimum length (number of edges) of a path from
$x$ to $y$, which defines the integer-valued \emph{graph metric.} A 
\emph{geodesic path} is one whose length is the distance between its
endpoints.

We select a finite, non-empty subset $F$ of $\VV$ and consider the balls 
$B(F,n) = \{x : d(x,F) \le n\}$ (where $d(x,F) = \min \{ d(x,y) : y \in F\}$).
If we delete $B(F,n)$ then
the induced graph $\XX \setminus B(F,n)$ will fall apart into a finite
number of connected components, called \emph{cones} with respect to $F$. 
Each cone is a labelled,
symmetric graph $C$ with the \emph{boundary} $\bd C$ consisting of 
all vertices $x$ in $C$ having a neighbour outside $C$ (i.e., in $B(F,n)$). 
%(equivalently, $d(x,o)=n+1$). 
%We also consider the whole of $\XX$ as a cone with boundary $F$.

The following notion was introduced in \cite{MS2} for symmetric,  
labelled graphs and $F = \{o\}$.

\begin{dfn}\label{def:cfgraph}
The graph $\XX$ is called context-free with respect to $F$ if there is 
only a finite number of  isomorphism types of the cones with respect to $F$ as 
labelled graphs with boundary.
\end{dfn}

This means that there are finitley many cones $C_1, \dots, C_r$
(generally with respect to different radii $n$) such that for each cone $C$, 
we can fix a bijection $\phi_C$ from (the vertex set of) $C$ to precisely 
one of the $C_i\,$, this bijection sends $\bd C$ to 
$\bd C_i$, and $(x,a,y)$ is an edge with both endpoints in $C$ if and only 
if its image$\bigl(\phi_C(x),a,\phi_C(y) \bigr)$ is an edge of $C_i$. 
In this case, we say that $C$ is a cone of \emph{type} $i$.

Generally, as in \cite{MS2}, we are interested in the case when $F = \{o\}$
(or any other singleton), but there is at least one point where it will be 
useful to admit arbitrary finite, non-empty $F$. 

Another natural notion of context-freeness of $X$ with respect to $o$
is to require that the language $L_{o,o}(X)$ is context-free. We shall see 
that for deterministic, symmetric graphs this is equivalent with
context-freeness with respect to $o$ in the sense of Definition 
\ref{def:cfgraph}. One direction of this equivalence is practically contained 
in \cite{MS2}, but not stated explicitly except for the case of Cayley 
graphs of groups. 
The other direction (that context-freeness of $L_{o,o}$ implies that of
the graph) is shown in \cite{MS2} only for Cayley graphs of groups, which is 
substantially simpler than the general case treated below in Theorem 
\ref{thm:characterize2}.
% and follows via a tour through the structure theory of groups. 

\begin{thm}\label{thm:characterize1} 
If the symmetric, labelled graph $(\VV,E,\ell)$ with
label alphabet $\Si$ is context-free with respect to the finite, non-empty 
set $F \subset \VV$, then $L_{x,y}$ is a context-free language for all
$x,y \in \VV$. Furthermore, if the graph $\XX$ is deterministic, then so is 
the context-free language $L_{x,y}\,$.   
\end{thm}

\begin{proof}
Just for the purpose of this proof, we write $x_0\,, y_0$ instead of $x,y$
for the vertices for which $L_{x_0,y_0}$ will be shown to be context-free.
We may assume without loss of generality
that $x_0\,, y_0$ in $F$. Indeed, if this is not the case, then we
can replace $F$ by $F' = B(F,n)$, which contains $x_0$ and $y_0$ when $n$ is
sufficiently large. The cones with respect to $F'$ are also cones
with respect to $F$, so that $\XX$ is also context-free with respect to $F'$.  

\smallskip

Similarly to \cite[Lemma 2.3]{MS2}, we construct a deterministic pushdown 
automaton that accepts $L_{x_0,y_0}\,$. 

%The following contstruction of the \emph{graph of cone types} $\Gamma$
%will also be useful in  Section~\ref{sec:sensitive}.

\smallskip

We consider also the whole graph $\XX$ as a cone 
$C_0$ with boundary $F$, which we keep apart from the other representatives
$C_1, \dots, C_r$ of cones. 

If $C$ is a cone, then as a component of $\XX \setminus B(F,n)$ for some 
$n \ge 0$ it must be a \emph{successor} of another cone $C^-$.
The latter is the unique component of $\XX \setminus B(F,n-1)$
that contains $C$, when $n \ge 1$, while it is $C_0 =\XX$ when $n=0$. We also
call $C^-$ the \emph{predecessor} of $C$.

Different cones of type $j \in \{ 1, \dots, r\}$ may have predecessors 
of different types. Conversely, a cone $C$ of type $i\in \{ 0, \dots, r\}$ 
may have  none, one or more than one successors of type $j$, and the 
number $d_{i,j}$ of those 
successors depends only on $i$ and $j$. In the representative cone $C_i\,$,
we choose and fix a numbering of the distinct successors of type $j$ as 
$C_{i,j}^k\,$, $k=1, \dots, d_{i,j}\,$. If $C$ is any cone with type $i$
then we use the isomorphism $\phi_C: C \to C_i$ to transport this numering to
the successors of $C$ that have type $j$, which allows us to identify the
$k$-th successor of $C$ with type $j$.  

One can visualize the cone structure by a finite,
oriented graph $\Gamma$ with multiple edges and root $0$: the vertex set is 
the set of cone types $i\in \{ 0, \dots, r\}$, and there are 
$d_{i,j}$ oriented edges, which we denote by $t_{i,j}^k\,$ 
($k=1, \dots, d_{i,j}$) from vertex $i$ to vertex $j$ ($i \ge 0$, $j \ge 1$).    

\smallskip
   
Every vertex $x$ of $\XX$ belongs to the boundary of precisely one cone 
$C=C(x)$ with respect to $F$. We define the \emph{type} $i$ of $x$ as the type of $C(x)$.
Under the mapping $\phi_C$, our $x$ corresponds to precisely one
element of $\bd C_i$. We write $\phi(x)$ for that element, without 
subscript $C$, so that $\phi$ maps $\VV$ onto $\bigcup_i \bd C_i\,$. 
In particular, $\phi(x)=x$ for every $x \in F$.

Let $y \in \VV \setminus F$ with type $j$. Then there is $i$ (depending on $y$)
such that every neighbour $x$ of $y$ with $d(x,F) = d(y,F)-1$ has type $i$, 
and there is precisely one successor cone $C_{i,j}^k$ of $C_i$ that 
contains $\phi_{C(x)}(y)$. In this case, we write $\tau(y) = t_{i,j}^k\,$, 
the \emph{second order type} of $y$. Compare with \cite{MS2}.
If $y'$ is such that $C(y')=C(y)$ then $\tau(y') = \tau(y)$.

We now finally construct the required pushdown automaton $\Acal$. 
(Comparing with 
\cite{MS2}, we use more states and stack symbols, which facilitates
the description.) The set of states and stack symbols are 
$$
\Q = \biguplus_{i=0}^r \bd C_i \AND 
\Z = F \cup \bigl\{ t_{i,j}^k : i = 1, \dots, r\,,\; j=0, \dots, r\,,\;
k=1, \dots, d_{i,j} \bigr\}\,.
$$
(When $d_{i,j}=0$ then there is no $t_{i,j}^k$.)
Note that both sets contain $F$. In order to generate the language 
$L_{x_0,y_0}\,$, where $x_0, y_0 \in F$, then we use $x_0$ as the initial state
and $y_0$ as the (only) final state. We describe the transition function,
which -- like $\Q$ and $\Z$ -- does not depend on $x_0, y_0\,$.

We want to read an input word, which has to correspond to the label starting at
$x_0$. Inside the subgraph of $\XX$ induced by $F$, our $\Acal$ behaves just like
that subgraph, seen as a finite automaton. 

Outside of $F$, it works as follows.
At the $m$-th step, the automaton will be in a state that descibes
the $m$-th vertex, say $x$, of that path, by identifying $x$ as above
with the element $\phi(x)$ of $C_j$, where $j$ is the type of $x$. 
The current stack symbol
is of the form $t_{i,j}^k$ and serves to recall that $x$ lies in the $k$-th
successor cone of type $j$ of a cone with type $i$. If the next vertex along 
the path, say $y$, satisfies $d(y,F) = d(x,F)+1$, and $y$ has type $j'$ then
the state is changed to $\phi(y) \in C_{j'}\,$, and
the symbol $t_{j,j'}^{k'} = \tau(y)$ is added to
the stack.  If $d(y,F) = d(x,F)$, then only the state is changed from
$\phi(x)$ to $\phi(y)$.
Finally, if $d(y,F) = d(x,F)-1$ then the new state is again 
$\phi(y)$, while the top symbol in the stack is deleted.   
Formally, we get the following list of transition rules.
$$
\begin{array}{lrcl}
\hspace*{.3cm}\text{If}\;x \in F = \Q \cap \Z\,:&&&\\[3pt]  
&\hspace*{-2.5cm}\de(x,a,x) 
&\!\!\!\!=\!\!\!\!& \bigl\{ (y,y) : (x,a,y) \in E\,,\;y \in F \bigr\}\\[3pt]
&&&\;\cup\; 
\bigl\{\bigl(\phi(y),x\tau(y)\bigr) : (x,a,y) \in E\,,\;d(y,F) = 1
\bigr\}.\\[3pt]
\hspace*{.3cm}\text{If}\;x \in \VV \setminus F\,:&&& \\[3pt]  
&\hspace*{-2.5cm}\de\bigl(\phi(x),a,\tau(x)\bigr) 
&\!\!\!\!=\!\!\!\!& \bigl\{ \bigl(\phi(y),a,\tau(x)\tau(y)\bigr) 
   : (x,a,y) \in E\,,\;d(y,F) = d(x,F)+1 \bigr\}\\[3pt]
&&&\;\cup\; 
\bigl\{\bigl(\phi(y),\tau(y)=\tau(x)\bigr) : (x,a,y) \in E\,,\;d(y,F) = d(x,F)
\bigr\}\\[3pt]
&&&\;\cup\; 
\bigl\{\bigl(\phi(y),\ep\bigr) : (x,a,y) \in E\,,\;d(y,F) = d(x,F)-1
\bigr\}
\end{array}
$$ 
This is a finite collection of transitions, since $\phi(\cdot)$ and 
$\tau(\cdot)$ can take only finitely many different values. 

In view of the above explanations, $\Acal$ accepts $L_{x_0,y_0}\,$.
Also, when the graph $\XX$ is deterministic, then so is $\Acal\,$. 
\end{proof}

Before proving a converse of Theorem \ref{thm:characterize1}, we first 
need some preliminaries, and start by recalling a fact proved in 
\cite{MS1} and \cite{MS2}, see also {\sc  Woess} \cite{Wtrees} and  
{\sc Berstel and
Boasson}~\cite{BeBo}.

\begin{lem}\label{lem:diam}
If $L_{o,o}$ is context-free then there is a constant $M$ such that for 
each cone $C$ with respect to $o$, one has
$\;\diam(\bd C) \le M\,.$
\end{lem}

(The diameter is of course taken with respect to the graph metric.)
We shall see below how to deduce this, but it is good to know it in advance.

A context-free grammar $\Ccal = (\V,\Si,\Pb,S)$ is said to have
\emph{Chomsky normal form (CNF)}, if (i) every production rule is of the form
$T \vdash U\hat U$ or $T \vdash a$, where $U, \hat U \in \V$ (not necessariliy
distinct), resp. $a \in \Si$, and (ii) if $\epsilon \in L(\Ccal)$,
then there is the rule $S \vdash \epsilon$, and $S$ is not contained in
the right hand side of any production rule.

With a slight deviation from \cite{MS1}, we associate with each 
$w=a_1 \cdots a_n \in L(\Ccal)$,  $n \ge 2$ a labelled (closed)
\emph{polygon} $\PP(w)$ with length $n+1$. As a directed graph, it 
has distinct vertices $t_0, t_1, \dots, t_n$ and labelled edges 
$(t_{i-1},a_i,t_i)$, $i = 1, \ldots, n$, plus the edge $(t_0, S, t_n)$.
A \emph{(diagonal) triangulation}
of $\PP(w)$ is a plane triangulation of $\PP(w)$ obtained by inserting
only diagonals. Here, we specify those diagonals as oriented, labelled edges
$(t_i,T,t_j)$, where $t_i, t_j$ are not neighbours in $\PP(w)$ and $T \in \V$.
Furthermore, we will never have two diagonals between the same pair of 
vertices of  $\PP(w)$.
%, with possibly one exception of two diagonals of the form 
%$(t_0,T,t_k)$ and $(t_k,T,t_n)$ . 
(If $|w| \le 2$ we consider $\PP(w)$
itself triangulated.) The proof of the following Lemma
may help to make the construction of \cite{MS1} (used for Cayley graphs of
groups) more transparent. 

\begin{lem}\label{lem:triangulation} 
If $\Ccal = (\V,\Si,\Pb,S)$ is in CNF and $w=a_1 \cdots a_n \in L(\Ccal)$ 
with $n \ge 2$ then there is a diagonal triangulation of $\PP(w)$ with the 
property that whenever $(t_i,T,t_j)$ is a diagonal edge, then $T$ occurs in
a derivation $S \thens w$, $j-i \ge 2$
and $T \thens a_{i+1} \cdots a_j$.  
\end{lem}

\begin{proof} We start with a fixed derivation $S \thens w$, and explain 
how to build up the triangles step by step. 
%We first consider $t_0$ and $t_n$ as distinct, 
Suppose that $T \in \V$ occurs in our derivation, and that we have a 
``sub-derivation''
$T \vdash U \hat U \thens a_{i+1}\cdots a_k$, where $U, \hat U \in \V$.
Then there is $j \in \{i+1, \dots, k-1\}$ such that
$U \thens a_{i+1}\cdots a_j$ and $\hat U \thens a_{j+1}\cdots a_k$.
In this case, we draw a triangle with three oriented, labelled edges, namely
the \emph{`old'} edge $(t_i, T, t_k)$ and the two \emph{`new'} edges  
$(t_i, U, t_j)$ and $(t_j, \hat U, t_k)$.

If we have the derivation $S \thens a_1 \cdots a_n$, then it uses
successive steps of the form $T \vdash U \hat U$ with 
$U \hat U \thens a_{i+1}\cdots a_k$  
as above. We work through these steps one after the other, starting
with $S \vdash T_1 \hat T_1$, where $T_1 \thens a_1 \dots a_k$ and
$\hat T_1 \thens a_{k+1}\cdots a_n$. The first triangle has the `old' edge
$(t_0,S,t_n)$ and the `new' edges $(t_0, T_1, t_k)$ and 
$(t_k, \hat T_1, t_n)$.

At any successive step, we take one of the `new' edges $(t_i,T,t_k)$,
where $k-i \ge 2$ and proceed as explained at the beginning, so that we
add two `new' edges that make up a triangle together with $(t_i,T,t_k)$,
which is then declared `old'. We continue until all derivation steps
of the form $T \vdash U \hat U$ in our derivation $S \thens w$ are exhausted.
At this point, we have obtained a tiling of triangles that constitute 
a diagonal triangulation of its outer polygon, whose edges have the
form $(t_0,S,t_n)$ and $(t_{i-1},U_i,t_i)$ with $U_i \in \V$, $i=1, \ldots, n$.
The only steps of our derivation that we have not yet considered are
the terminal ones $U_i \vdash a_i$. Thus, we conclude by replacing the label
$U_i$ of $(t_{i-1},U_i,t_i)$ by $a_i$.
\end{proof}

The construction is best understood by considering an example: suppose
our rightmost derivation is
$$
\begin{array}{rll}
S \vdash T_1 \hat T_1 
&\then T_1 (T_2 \hat T_2) &\then T_1 (T_2 (T_3 \hat T_3)) \\
&\then T_1 (T_2 (T_3 a_6)) &\then T_1 (T_2 ((T_4 \hat T_4)a_6)) \\ 
&\then T_1 (T_2 ((T_4 a_5)a_6)) &\then T_1 (T_2 ((a_4 a_5)a_6)) \\
&\then T_1 (a_3 ((a_4 a_5)a_6)) &\then T_1 (a_3 ((a_4 a_5)a_6)) \\
&\then (T_5 \hat T_5) (a_3 ((a_4 a_5)a_6)) 
&\then (T_5 a_2) (a_3 ((a_4 a_5)a_6)) \\
&\then (a_1 a_2) (a_3 ((a_4 a_5)a_6))&
\end{array}
$$
(We have inserted the parentheses to make the rules that we used in
each step more visible.) 
The associated triangulation is as follows. 

%\vspace{1cm}

\begin{center}
\begin{picture}(220,220)
\put(70,20){\circle{28}}
\put(150,20){\circle{28}}
\put(30,80){\circle{28}}
\put(40,150){\circle{28}}
\put(190,80){\circle{28}}
\put(180,150){\circle{28}}
\put(110,190){\circle{28}}

\put(66,18){$t_5$}
\put(146,18){$t_4$}
\put(26,78){$t_6$}
\put(36,148){$t_0$}
\put(186,78){$t_3$}
\put(176,148){$t_2$}
\put(106,188){$t_1$}

\put(62,178){$a_1$}
\put(142,178){$a_2$}
\put(22,112){$S$}
\put(191,114){$a_3$}

\put(176,44){$a_4$}
\put(33,44){$a_6$}
\put(103,10){$a_5$}
\put(106,86){$\hat{T}_2$}

\put(96,120){$\hat{T}_1$}
\put(102,156){${T}_1$}

\put(116,53){${T}_3$}
\put(84,20){\line(1,0){52}}
\put(136,20){\vector(-1,0){32}}
\put(58,28){\line(-1,2){20}}
\put(162,28){\line(1,2){20}}
\put(58,28){\vector(-1,2){11}}
\put(182,68){\vector(-1,-2){12}}

\put(31,94){\line(1,6){7}}
\put(35.2,117){\vector(1,3){0.1}}
\put(191,94){\line(-1,6){7}}
\put(187.7,115){\vector(1,-3){0.1}}

\put(54,150){\line(1,0){112}}
\put(54,150){\vector(1,0){56}}

\put(44,80){\line(1,0){132}}
\put(166,80){\vector(-1,0){56}}

\put(178,73){\line(-2,-1){96}}
\put(178,73){\vector(-2,-1){56}}

\put(166,145){\line(-2,-1){123}}
\put(166,145){\vector(-2,-1){65}}

\put(49,160){\line(2,1){48}}
\put(49,160){\vector(2,1){24}}
\put(122,184){\line(2,-1){48}}
\put(122,184){\vector(2,-1){25}}
%\put(48,55){\vector(-1,0){1}}
%\put(48,55){\vector(-1,0){1}}
%\put(152,85){\vector(1,0){1}}
%\put(148,55){\vector(-1,0){1}}
\end{picture}
\end{center}
\begin{center}
\emph{Figure 2.}
\end{center}
\vspace{.5cm}

The variables of the terminal rules 
$T_5\vdash a_1\,$, $\hat T_5 \vdash a_2\,$, $T_2 \vdash a_3\,$,
$T_4\vdash a_4\,$, $\hat T_4\vdash a_5$ and $\hat T_3\vdash a_6$ are not 
visible in this figure (but we might add them to the boundary edges).
Apart from this, one can read the derivation
$S \thens w$ from the diagonalization in a similar way as it can
be read from the so-called derivation tree (see e.g. \cite[\S 1.6]{Ha}
for the latter.) 

The following goes back to \cite{MS1} in the case of (Cayley graphs of) finitely
generated groups.

\begin{lem}\label{lem:diagonal} 
Let $\Ccal = (\V,\Si,\Pb,S)$ be in CNF and $L(\Ccal) = L_{x,y}(\XX)$, where
$\XX$ is a deterministic, symmetric graph. If 
$w=a_1 \cdots a_n \in L_{x,y}(\XX)$ and $(t_i,T,t_j)$ is a diagonal edge
in a triangulation of $\PP(w)$ as in Lemma \ref{lem:triangulation}, then
the vertices $\bar x = x^{a_1 \cdots a_i}$ and $\bar y = x^{a_1 \cdots a_j}$
of $\XX$ satisfy $d(\bar x,\bar y) \le m(T)$, where
\begin{equation}\label{eq:mT}
d(\bar x,\bar y) \le m(T) = \min\{ |w| : w \in L_T \}\,.
\end{equation}
\end{lem}

\begin{proof} Since $\XX$ is deterministic, Lemma \ref{def:pathfromx} 
implies that $\pi_x(w)$ exists as the unique path with initial vertex $x$
and label $w$. In particular, $\bar x$ and $\bar y$ lie on that path.
Furthermore, we have $\bar y = y^{-a_{j+1} \cdots a_n}$.
 
Now let $v \in L_T$ with $|v|=m(T)$. Then by Lemma 
\ref{lem:triangulation}, $T$ arises in a
derivation $S \thens a_1 \cdots a_i T a_{j+1} \cdots a_n \thens w$.
%, and $T \thens a_{i+1} \cdots a_j$. 
But then we also have $S \thens a_1 \cdots a_i v a_{j+1} \cdots a_n$,
a word in $L_{x,y}$. By Lemma \ref{def:pathfromx}, again using that $\XX$ is
symmetric and deterministic,
$\bar x^v = y^{-a_{j+1} \cdots a_n} = \bar y$. Therefore,
$\bar x$ and $\bar y$ are connected by a path with label $v$. Its length
is $m(T)$.  
\end{proof}  

\begin{thm}\label{thm:characterize2} 
Let $(\VV,E,\ell)$ be a fully deterministic, symmetric graph  with
label alphabet $\Si$ and root $o$. If $L_{o,o}$
is a context-free language, then $\XX$ is a context-free graph with respect
to $o$, and in particular, $L_{o,o}$ is deterministic.
\end{thm}

\begin{proof}  
There is a reduced grammar $\Ccal = (\V,\Si,\Pb,S)$ in CNF that
generates $L_{o,o}\,$. Each of the languages $L_T$, $T \in \V$, is non-empty,
only $L_S$ contains $\epsilon$, and we define
\begin{equation}\label{eq:m}
m=\max \{ m(T) : T \in \V \}\,,
\end{equation}
where $m(T)$ is as in \eqref{eq:mT}.

Let $C$ be a cone with respect to $o$ such that $k=d(o,\bd C) > m$.

\smallskip

\noindent
\emph{Construction of $\wt D(C)$.}
We define $D(C)$ as the subgraph of $\XX$ induced by all vertices 
$y \in \VV$ with
$$
d(o,x) = d(o,y)+d(x,y) \AND d(x,y) \le m \quad\text{for some}\;\;x \in \bd C\,.
$$
In particular, $y$ lies on some geodesic path from $o$ to $\bd C$. 

Now let $x_1, x_2 \in \bd C$, and consider some path $\pi \in  \Pi_{x_1,x_2}(C)$
(i.e., it lies in $C$). Choose a geodesic path $\pi_1$ from $o$ to
$x_1$ and a geodesic path $\pi_2$ from $x_2$ to $o$. Then we can concatenate
the three paths to a single path $\pi_1\pi\pi_2 \in \Pi_{o,o}$. Its label
is the word $w=\ell(\pi_1)\ell(\pi)\ell(\pi_2) \in L_{o,o}$. Set $n = |w|$
and write 
$$
w = (a_1 \cdots a_k)(a_{k+1} \cdots a_{n-k})(a_{n-k+1} \cdots a_n)\,
$$
where the 3 pieces in the parentheses are (in order) 
$\ell(\pi_1)$, $\ell(\pi)$ and $\ell(\pi_2)$. The words $\ell(\pi_1)$, 
$\ell(\pi)$ and $\ell(\pi_2)S$ are the labels of three consecutive arcs that  
fill the boundary of the polygon $\PP(w)$. (To be precise, along the last 
edge of the $3^{\text{rd}}$ arc,
we are reading the label $S$ in the reversed direction.) 
By \cite[Lemma 5]{MS1}, its 
triangulation has a triangle which meets each of those arcs. (It may also occur
that one corner of the triangle meets two arcs.) Thus, there are  
$i \in \{ 0,\dots, k \}$ and $i' \in \{k, \dots, n-k\}$ such that the 
vertices $t_i$ and $t_{i'}$ of $\PP(w)$ lie on that triangle. 
They correspond to the
vertices $y_1 = o^{a_1 \cdots a_i}$ and $y' =o^{a_1 \cdots a_{i'}}$ of $\XX$. 
We either have $i' - i \le 1$, or else a diagonal $(t_i,U,t_{i'})$ is a side of 
our triangle. By Lemma \ref{lem:diagonal}, we get 
$d(y_1,y') \le m(U) \le m$. Thus
%Since $\pi_1$ is geodesic, $d(o,o^{a_1 \cdots a_i}) = i$, and 
$k \le i' \le d(o, y') \le i + m$, that is, $i \ge k-m > 0$. In particular, 
$t_i$ does not lie on the third arc. In the same way, there is
$j \in \{n-k, \dots, n-k+m\}$ (and not larger) such that $t_j$ is a corner
of our tiangle. This yields that there must be a ``true'' diagonal 
$(t_i,T,t_j)$ of $\PP(w)$. 
We set $v_1 = a_{i+1} \cdots a_k$ and $v_2 =a_{n-k+1} \cdots a_j$, so that
$x_1 = y_1^{v_1}\,$, and let $y_2 = x_2^{a_{n-k+1} \cdots a_j}$. 
The points $y_1$ and $y_2$ are in $D(C)$, and by Lemma \ref{lem:triangulation},
$T \thens v_1\ell(\pi)v_2\,$. 

[It is here that we can see Lemma \ref{lem:diam}, since we deduce
that $d(x_1,x_2) \le 3m$ for all $x_1,x_2 \in \bd C$.]

By Lemma \ref{lem:triangulation}, we also have 
$$
S \thens a_1\cdots a_i T a_{j+1} \cdots a_n\,,
$$
so that $v \in L_T$ implies $a_1\cdots a_i v a_{j+1} \cdots a_n \in L_{o,o}$ and
consequently $v \in L_{y_1,y_2}$, that is, $y_1^v=y_2$.

We now insert into $D(C)$ the additional labelled edge $(y_1, v_1Tv_2, y_2)$,
whose label is the word $v_1Tv_2 \in \Si^*\V\Si^*$. 
We insert all diagonals of the same type that can be obtained in the
same way, and write $\wt D(C)$ for the resulting ``edge-enrichment'' of
$D(C)$. 

Subsuming, we have an edge $(y_1, v_1Tv_2, y_2)$ in $\wt D(C)$ if and only
if the following properties hold.
\begin{itemize}
\item $\;|v_i| \le m$ ($i=1,2$) and $T \in \V\,$,
\item the path with label $v_1$ starting at $y_1$ and ending at 
$x_1=y_1^{v_1} \in \bd C$ is part of a geodesic from $o$ to $x_1\,$,
\item the path with label $v_2$ starting at $x_2=y_2^{-v_2} \in \bd C$ and 
ending at $y_2$ is part of a geodesic from $x_2$ to $o\,$, and
\item there is a path $\pi$ in $C$ from $x_1$ to $x_2$ such that
$T \thens  v_1\ell(\pi)v_2$, 
\item if $T \thens v \in \Si^*$ then $v$ is the label of a path in
$\Pi_{y_1,y_2}$.
\end{itemize}

Now, there are only finitely many cones $C$ with respect to $o$ with $d(\bd C, o) \le m$.
On the other hand, for all cones $C$ with $d(\bd C, o) \ge m$, there is a 
bound on the number of vertices of $\wt D(C)$,  as well as on the number of 
possible labels on its edges. In particular, there are only finitely many
possible isomorphism types of the labelled graphs $\bigl(\wt D(C),\bd C\bigr)$ 
with ``marked'' boundary $\bd C \subset \wt D(C)$. 

\smallskip

We now suppose that $C$ and $C'$ are two cones at distance $\ge m$ from $o$,
such that $\bigl(\wt D(C),\bd C\bigr)$ and $\bigl(\wt D(C'),\bd C'\bigr)$ 
are ismorphic. We claim that $C$ and $C'$ are isomorphic, and this will
conclude the proof that there are only  finitely many isomorphism types
of cones with respect to $o$. 

Let $\phi: \wt D(C) \to \wt D(C')$ be an isomorphism with 
$\phi(\bd C) = \bd C'$, and $\phi'$ its inverse
mapping. We extend $\phi$ to a mapping from $C$ to $C'$, also denoted $\phi$.

\smallskip

\noindent
\emph{Claim 1.} Let $x \in \bd C$ and $v \in \Si^+$ such that the
path $\pi_x(v)$ lies in $C$ and meets $\bd C$ only in its initial point
$x$. Then the path $\pi_{x'}(v)$ lies in $C'$ and meets $\bd C'$ only in 
its initial point $x'= \phi(x) \in \bd C'$.

\smallskip

\noindent
\emph{Proof.}
If $a$ is the initial letter of $v$ then (always using the notation 
of Definition \ref{def:pathfromx}) the first edge of $\pi_x(v)$ 
is $(x,a,x^a)$. We now consider the 
path $\pi_{x'}(v)$ with label $v$ starting at $x' \in \bd C'$. 
We first claim that the latter lies in $C'$ and only its initial point 
$x'$ is in $\bd C'$. Let $(x',a, (x')^a)$ be the first edge of the path. 
Then $(x')^a$ cannot lie in $\wt D(C')$, since otherwise  
$(x,a,x^a) = \bigl(\phi'(x'),a,\phi'(x')^a\bigr)$
would be an edge in $\wt D(C)$, a contradiction. Thus, the path $\pi_{x'}(v)$ 
goes at least initally into $C' \setminus \bd C$.

So now suppose that
$\pi_{x'}(v)$ ever returns to $\bd C'$, and let $\pi'$ be its initial part
up to the first return. Then $v' = \ell\bigl(\pi_{x'}(v)\bigr)$ is an initial
part of $v$ with $|v'| \ge 2$, and $\pi'$ is a path within $C'$ from $x_1'=x'$ to 
$x'_2 = (x')^{v'} \in \bd C'$. But then, by construction, $\wt D(C')$
must contain an edge $(y_1',v_1Tv_2,y_2')$ such that  $x_1'=(y_1')^{v_1}$,
$y_2'=(x_2')^{v_2}$, and $T \thens v_1v'v_2$. Using the isomorphism
$\phi': \wt D(C') \to \wt D(C)$, we set $y_i = \phi'(y_i')$, $i=1,2$,
and $x_2 = \phi'(x_2') \in \bd C$. We have of course $x_1 = \phi'(x_1')$.  
Now we must have the edge $(y_1,v_1Tv_2,y_2)$ in $\wt D(C)$. But then
$v_1v'v_2 \in L_{y_1,y_2}$, and consequently $v' \in L_{x_1,x_2}$,
that is, $x_1^{v'} \in \bd C$. But this contradicts the fact that $\pi_x(v)$
meets $\bd C$ only in its initial point. We conclude that also the path
$\pi_{x'}(v)$ lies in $C'$ and meets $\bd C'$ only in its inital point,
and Claim 1 is verified.

\smallskip

Now let $z \in C \setminus \bd C$. Then there are $x \in \bd C$ and 
$v \in  \Si^+$ such that $z=x^v$ and the path $\pi_x(v)$  from $x$ to $z$ meets $\bd C$ only 
in its initial point $x$. By Claim 1, the analogous statement holds
for the path $\pi_{x'}(v)$ in $C'$, where $x' = \phi(x)$. The only 
choice is to define $\phi(z) = z' = (x')^v$, which lies in $C' \setminus \bd C'$
as required. We have to show that $\phi$ is well-defined.
This will follow from the next claim.

\smallskip

\noindent
\emph{Claim 2.} Let $x_1, x_2 \in \bd C$, $v, w \in \Si^+$ such
that the paths $\pi_{x_1}(v)$ and $\pi_{x_2}(w)$ lie in $C$, meet $\bd C$
only in their initial points and end at the same point of $C \setminus \bd C$.
Then, setting $x_i' = \phi(x_i)$, also $\pi_{x_1'}(v)$ and $\pi_{x_2'}(w)$ 
end at the same point of $C' \setminus \bd C'$.

\smallskip

\noindent
\emph{Proof.} Let $w^{-1}$ be the ``inverse'' of $w$, as defined
in Definition \ref{def:pathfromx}. 
%\begin{equation}\label{eq:inverse}
%\text{if} \quad w = a_1 \cdots a_k \text{then} \quad 
%w^{-1} = a_k^{-1} \cdots a_1^{-1}\,.
%\end{equation}
Then $x_2^{-w^{-1}} = x_2^w$, and $vw^{-1}$
is the label of the path from $x_1$ to $x_2$ that we obtain by first
following $\pi_{x_1}(v)$ and then the ``inverse'' of $\pi_{x_2}(w)$. It lies
entirely in $C$, and only its endpoints are in $\bd C$. 
By construction, $\wt D(C)$ has an edge $(y_1,v_1Tv_2,y_2)$
such that $y_1^{v_1} = x_1$, $x_2^{v_2} = y_2$ and $T \thens v_1v w^{-1} v_2$.
We set $y_i' = \phi(y_i)$, $i=1,2$. Then
$(y_1',v_1Tv_2,y_2')$ is an edge of $\wt D(C')$. Therefore 
$v_1v w^{-1} v_2 \in L_{y_1',y_2'}$. But this implies that $v w^{-1}$ is the
label of a path from $x_1'$ to $x_2'$, and we know from Claim 1 that it
lies in $C$ and has only its endpoints in $\bd C$. Thus 
$(x_1')^{v} = (x_2')^{-w^{-1}} = (x_2')^{w}$,
and Claim 2 is true.

Thus, $\phi$ is well defined, and the same works of course also for $\phi'$
by exchanging the roles of $C$ and $C'$. 

\smallskip
\noindent
\emph{Claim 3.} The map $\phi: C \to C'$ is bijective. 

\smallskip

\noindent
\emph{Proof.} We know that $\phi: \bd C \to \bd C'$ is bijective and
that $\phi(C \setminus \bd C) \subset C' \setminus \bd C$. Let 
$z \in C \setminus \bd C$, and let $x \in \bd C$, $v\in \Si^+$ such
that $\pi_x(v)$ is a path from $x$ to $z$ that intersects $\bd C$ only at
the initial point. Setting $x' = \phi(x)$, $z'=\phi(z)$, we know from 
the construction of $\phi$ and Claim 1 that $\pi_{x'}(v)$ is a path in
$C'$ from $x'$ to $z'$ that meets $\bd C'$ only in its initial point.
Now the way how $\phi'$ is constructed yields that $\phi'(z') = z$.
Therefore $\phi' \, \phi$ is the identity on $C$. Exchanging roles,
we also get the $\phi \, \phi'$ is the identity on $C'$. This proves
Claim 3.

\smallskip

It is now immediate from the construction that $\phi$ also preserves the 
edges and their labels, so that it is indeed an isomorphism between
the labelled graphs $C$ and $C'$ that sends $\bd C$ to $\bd C'$. 
\end{proof}

\cite[Cor. 2.7]{MS2} says that
if a symmetric labelled graph is context-free with respect
to one root $o$, then it is context-free with respect to any other
vertex chosen as the root $x$. In view of Theorems \ref{thm:characterize1}
and \ref{thm:characterize2},
this is also obtained from the following, when the graph is fully deterministic.

\begin{cor}\label{cor:Lxy}
Let $(\VV,E,\ell)$ be a fully deterministic, strongly connected graph with
label alphabet $\Si$. If $L_{o,o}$ is context-free then
$L_{x,y}$ is deterministic context-free for all $x, y \in \VV$.  
\end{cor}

%\begin{proof} If $L_{o,o}$ is context-free then $\XX$ is a context-free graph
%with respect to $F = \{o\}$. By Theorem \ref{thm:characterize1}, 
%$L_{x,y}$ is deterministic context-free for all $x, y \in \VV$. 
%\end{proof}

Theorems \ref{thm:characterize1} and \ref{thm:characterize2}, together with
Lemma \ref{lem:translator} also imply the following.

\begin{cor}\label{cor:pairs}
Let $G$ be a finitely generated group and $K$ a subgroup.\\[5pt]
{\rm (a)} The pair $(G,K)$ is context-free if and only if for any symmetric
$\psi: \Si \to G$, the Schreier graph $\XX(G,K,\psi)$ is a context-free
graph. In this case, the language $L(G,K,\psi)$ is deterministic for
every (not necessarily symmetric) semigroup presentation 
$\psi:\Si \to G$.\\[5pt]
{\rm (b)} If $(G,K)$ is context-free, then also $(G,g^{-1}Kg)$ is
context-free for every $g \in G$. 
\end{cor}

\begin{proof} (a) is clear. Regarding (b), for the Schreier graph 
$\XX(G,K,\psi)$, we have $L(G,K,\psi)=L_{o,o}$ and 
$L(G,g^{-1}Kg,\psi) = L_{x,x}$ with
$x= Kg$, $g \in G$. Thus, the statement follows from Corollary
\ref{cor:Lxy}.
\end{proof}

\begin{lem}\label{lem:normal} Let $G$ be a finitely generated group
and $K, H$ be subgroups with $K \le H$ and $[H:K] < \infty$.

If $(G,K)$ is context-free then also $(G,H)$ is context-free.
\end{lem}

\begin{proof}
In the context-free graph $\XX(G,K,\psi)$, consider
the finite set of vertices $F = \{ Kh : h \in H \}$, containing the
root vertex $o=o_K = K$. Then $L(G,H,\psi) = \bigcup_{x \in F} L_{o,x}$
is a finite (disjoint) union of context-free languages. Therefore it is
context-free by standard facts. 
\end{proof}

\begin{rmk}\label{rmk:project}
In terms of Schreier graphs, we
have the mapping $Kg \mapsto Hg$ which is a homomorphism of labelled
graphs from $X = X(G,K,\psi)$ onto $Y = X(G,K,\psi)$ which is finite-to-one.
The lemma says that in this situation, if $X$ is a context-free graph then
so is $Y$. We do not see an easy direct proof of this fact in terms of graphs,
the main problem being how the homomorphism $X \to Y$ interacts with the
isomorphisms between the cones of $X$ with respect to the set $F$.
On the other hand, reforomulating this in terms of the associated ``path 
languages'' with the help
of theorems \ref{thm:characterize1} and \ref{thm:characterize2}, it 
has become straightforward.
\end{rmk}

The converse of Lemma  \ref{lem:normal} is not true, that is, when $(G,H)$
is context-free and $[H:K] < \infty$ then $(G,K)$ is not necessarily context-
free. See Example \ref{exa:fin-ind} in the last section. However, we have
the following.

\begin{lem}\label{lem:Kfinite} If $K$ is a finite subgroup of $G$ then
$(G,K)$ is context-free if and only if $G$ is a context-free (i.e. virtually
free) group.
\end{lem}

\begin{proof} Fix $\Si$ and $\psi$. Let $X = X(G,\psi)$ be the associated
Cayley graph of $G$, and $Y = X(G,K,\psi)$. We let $o$ be the root of $Y$,
that is, $o=K1_G$ as an element of $Y$ (a coset). The group $K$ acts on $X$ by
automorphisms of that labelled graph. It leaves the set $F=K$ (now as a
set of vertices of $X$) invariant. The factor graph of $X$ by this action is
$Y$. Write $\pi$ for the factor mapping. It is $|K|$-to-one. Each cone of
$X$ with respect to $F$ is mapped onto a cone of $Y$ with respect to $o$,
and this mapping sends boundaries of cones of $X$ to boundaries of cones
of $Y$. By assumption,  $Y$ is a context-free graph. By Lemma \ref{lem:diam},
there is an upper bound on the number of elements in the latter boundaries.
Therefore there also is an upper bound on the number of elements of any of
the boundaries of the cones of $X$ with respect to $F$. 

Without going here into the details of the definition of the space of 
\emph{ends} of $X$, we refer to the terminology of {\sc Thomassen and 
Woess}~\cite{ThWo} and note that the above implies that all ends of $X$
are \emph{thin.} But then, as proved in \cite{ThWo}, $G$ must be a virtually
free group.
\end{proof}

One should not tend to believe that in the situation of the last lemma,
the Cayley graphs of $G$ are quasi-isometric with the Schreier graphs of
$(G,K)$. As a simple counter-example, take for $G$ the infinite dihedral group
$\langle a,b \mid a^2 = b^2 \rangle$ and for $K$ the 2-element subgroup
generated by $a$.

%%%%%%%%%%%%%%%%%%%%%%%%%%%%%%%%%%%%%%%%%%%%%%%%%%%%%%%%%%%%%%%%%%%%%%

\section{Covers and Schreier graphs}\label{sec:covers}
We assume again that $(\VV,E,\ell)$ is symmetric and fully deterministic.
Recall the involution $a \mapsto a^{-1} \ne a$ of $\Si$. 
A word in $\Si^*$ is called reduced if it contains no subword of the form
$a a^{-1}$, where $a \in \Si$. We write $\T_{\Si}$ for the set of all
reduced words in $\Si^*$. We can equip $\T_{\Si}$ with the structure of
a labelled graph, whose edges are of the form
\begin{equation}\label{eq:treeedges}
(v, a, w) \AND (w, a^{-1},v)\,,
\quad \text{where}\;\; v, w \in \T_{\Si}\,,\; a \in \Si\,,\; va=w\,.
\end{equation} 
Thus, the terminal letter of $v$ must be different from $a^{-1}$. 
Then $\T_{\Si}$ is fully deterministic, and it is a \emph{tree,} that is,
it has no closed path whose label is a (non-empty) reduced word.  
As the root of $\T_{\Si},$ we choose the empty word $\epsilon$.
Then $\T_{\Si}$ is the \emph{universal cover} of $\XX$. Namely, if we choose
(and fix) any vertex $o \in \XX$ as the root, then the mapping
\begin{equation}\label{eq:cover}
\Phi: \T_{\Si} \to \XX\,,\quad \Phi(w) = o^w\,,
\end{equation}
is a \emph{covering map:} it is a surjective homomorphism between labelled graphs
which is a local isomorphism, that is, it is one-to-one between the sets of
outgoing (resp. ingoing) edges of any element $w \in \T_{\Si}$ and its
image $\Phi(w)$. (Note that this allows the image of an edge to be a loop.) 
``Universal'' means that it covers every other cover of $\XX$, but this is
not very important for us. The property of $w \in \T_{\Si}$ to be reduced
is equivalent with the fact that the path $\pi_o(w)$ in $\XX$ is 
\emph{non-backtracking,} that is, it does not contain two consecutive 
edges which are the reversal of each other.

We now realize that $\T_{\Si}$ is the standard Cayley graph of the free group
$\F_{\Si}$, where $\Si$ is the set of free generators together with their
inverses. The group product is the following: if 
$v, w \in \T_{\Si} \equiv \F_{\Si}\,$, then $v \cdot w$ is obtained from
the concatenated word $vw$ by step after step deleting possible subwords 
of the form $aa^{-1}$ that can arise from that concatenation.
The group identity is $\epsilon$, and the inverse of $w$ is $w^{-1}$
as at the end of Definition \ref{def:pathfromx}.
With $\Phi$ as in \eqref{eq:cover}, let 
\begin{equation}\label{eq:subgp}
\K = \K(\XX) = \Phi^{-1}(o) = \{ w \in \T_{\Si}  : \pi_o(w) 
\; \text{is a closed path from $o$ to $o$ in $\XX$}\,\}\,.
\end{equation}
Then, under the indentification $\T_{\Si} \equiv \F_{\Si}\,$, we clearly
have that $\K$ is a subgroup of $\F_{\Si}\,$. The following is known,
see e.g. {\sc Lyndon and Schupp}~\cite[Ch. III]{LySc} or (our personal source)
{\sc Imrich}~\cite{Im}.

\begin{pro}\label{pro:Schreier}
The graph $\XX$ is the Schreier graph of the pair of groups 
$\bigl(\F_{\Si},\K(\XX)\bigr)$ with respect to the semigroup presentation
$\psi$ given by $\psi(a) = a\,$, $a \in \Si$.
\end{pro}
In $\psi(a) = a$, we interpret $a$ simultaneously as a letter from the alphabet
and as a generator of the free group.

Thus, in reality the study of context-free pairs of groups is the same
as the study of fully deterministic, \emph{symmetric} context-free graphs 
under a different viewpoint. 

The same is not true without assuming symmetry.
Indeed, given a semigroup presentation $\psi$ of $G$, for every $a \in \Sigma$
there must be $w_a \in \Sigma^*$ such $\psi(w_a) = \psi(a)^{-1}$, the inverse 
in $G$. But then in the Schreier Graph $X(G,K,\psi)$, for any subgroup 
$K$ of $G$, we have the following: if $(x,a,y) \in E$ then $y^{w_a} =x$, 
that is, there is the oriented path from $y$ to $x$ with label $w_a$. In 
a general fully deterministic graph this property does not necessarily hold,
even if it has the additional property that for each $a \in \Sigma$,
there is presicely one incoming edge with label $a$ at every vertex.
As an example, consider $X = \{x,y,z\}$, $\Sigma = \{a,b\}$ and
labelled edges $(x,a,y), (x,b,y), (y,a,z), (y,b,x), (z,a,x), (z,b,z)$.

We return to the situation of Proposition \ref{pro:Schreier}. As a subgroup 
of the free group, the group $\K(\XX)$ is itself free.
There is a method for finding a set of free generators. First recall
the notion of a \emph{spanning tree} of $\XX$. This is a tree $T$, which as
subgraph of $\XX$ is obtained by deleting  edges (but no vertices) 
of $\XX$. Every connected (non-oriented) graph has a spanning tree, for 
locally finite graphs it can be constructed inductively.
Now let $T$ be a spanning tree of $\XX$, and consider all edges of $\XX$ that
are not edges of $T$. They must come in pairs $(e, e^{-1})$. 
For each pair, we choose one of the two partner edges, and we write
$E_0$ for the chosen (oriented) edges. For each $e \in E_0$, we choose
non-backtracking paths in $T$ from $o$ to $e^-$ and from $e^+$ to $o$.
Together with $e$ (in the middle), they give rise to a non-backtracking path
in $\XX$ that starts and ends at $o$. Let $w(e)$ be the label on that path.
Then the following holds \cite{LySc}, \cite{Im}.

\begin{pro}\label{pro:generators}
As elements of $\F_{\Si}$, the $w(e)$, $e \in E_0$, are free generators of
$\K(\XX)$. 
\end{pro} 

\begin{cor}\label{cor:fingen}
Let $G$ be a virtually free group and $K$ a finitely generated subgroup.
Then $(G,K)$ is context-free.
\end{cor}

\begin{proof} Let $\F = \F_{\Si}$ be a free subgroup of $G$ with
finite index. Then $\K = K \cap \F$ is a free subgroup of $K$ with
$[K:\K] < \infty$. Since $K$ is finitely generated, also 
$\K$ is finitely generated.  
In the Schreier graph $\XX$ of $(\F,\K)$ with respect to the standard 
labelling by $\Si$, choose a spanning tree and remaining set $E_0$ of edges,
as described above.  Since all sets of free generators of 
$\K$ must have the same cardinality, $E_0$ is
finite. Thus, $\XX$ is obtained by adding finitely many edges to a tree.
If $o$ is the root vertex of $\XX$ and $n$ is the largest distance between
$o$ and an endpoint of some edge in $E_0$, then every cone $C$ of $\XX$
whith $d(\bd C, o) > n$ is a rooted, labelled tree that is isomorphic with
one of the cones of $\T_{\Si}$. Thus, the Schreier graph, resp.
$(\F,\K)$ are context-free. It now follows from Proposition \ref{pro:extend} 
and Lemma \ref{lem:normal} that also $(G,K)$ is context-free.
\end{proof} 

We remark here that one can always reduce the study of context-free pairs
to free groups and their subgroups. Given $(G,K)$, let $\F$ be a finitely 
generated free group that maps by a homomorphism onto $G$. Let $\K$
be the preimage of $K$ under that homomorphism. Then clearly $(G,K)$ is
context-free if and only $(\F,\K)$ has this property. (This reduction,
however, is not very instructive.)

Of course, there are context-free pairs with $G$ free
beyond the situation of Corollary \ref{cor:fingen}.

\begin{exa}\label{exa:comb}
Consider the free group $\F = \langle a,b \mid \ \rangle$ and
the subgroup $\K$ with the infinite set of free generators
$\{ a^kb^lab^{-l}a^{-k} : k, l \in \ZZ\,,\; l \ne 0 \}$. The associated
Schreier graph with respect to $\{a^{\pm 1}\,,b^{\pm 1}\}$
is the \emph{comb lattice}. 

\begin{center}
\begin{picture}(220,260)
\qbezier[200](70,170)(20,190)(40,205) 
\qbezier[200](70,170)(50,220)(35,200) 
\put(39,204){\vector(1,1){1}}
\put(39,196){$a$}

\qbezier[200](70,20)(20,40)(40,55) 
\qbezier[200](70,20)(50,70)(35,50) 
\put(39,54){\vector(1,1){1}}
\put(39,46){$a$}

\qbezier[200](70,70)(20,90)(40,105) 
\qbezier[200](70,70)(50,120)(35,100) 
\put(39,104){\vector(1,1){1}}
\put(39,96){$a$}

%%%%

\qbezier[200](170,170)(120,190)(140,205) 
\qbezier[200](170,170)(150,220)(135,200) 
\put(139,204){\vector(1,1){1}}
\put(139,196){$a$}

\qbezier[200](170,20)(120,40)(140,55) 
\qbezier[200](170,20)(150,70)(135,50) 
\put(139,54){\vector(1,1){1}}
\put(139,46){$a$}

\qbezier[200](170,70)(120,90)(140,105) 
\qbezier[200](170,70)(150,120)(135,100) 
\put(139,104){\vector(1,1){1}}
\put(139,96){$a$}
%%%

\qbezier[200](120,170)(70,190)(90,205) 
\qbezier[200](120,170)(100,220)(85,200) 
\put(89,204){\vector(1,1){1}}
\put(89,196){$a$}

\qbezier[200](120,20)(70,40)(90,55) 
\qbezier[200](120,20)(100,70)(85,50) 
\put(89,54){\vector(1,1){1}}
\put(89,46){$a$}

\qbezier[200](120,70)(70,90)(90,105) 
\qbezier[200](120,70)(100,120)(85,100) 
\put(89,104){\vector(1,1){1}}
\put(89,96){$a$}
%%%

\qbezier[200](220,170)(170,190)(190,205) 
\qbezier[200](220,170)(200,220)(185,200) 
\put(189,204){\vector(1,1){1}}
\put(189,196){$a$}

\qbezier[200](220,20)(170,40)(190,55) 
\qbezier[200](220,20)(200,70)(185,50) 
\put(189,54){\vector(1,1){1}}
\put(189,46){$a$}

\qbezier[200](220,70)(170,90)(190,105) 
\qbezier[200](220,70)(200,120)(185,100) 
\put(189,104){\vector(1,1){1}}
\put(189,96){$a$}

\put(-3,117){$\bf{\cdots}$}
\put(230,117){$\bf{\cdots}$}

\put(18,-11){$\bf{\vdots}$}
\put(68,-11){$\bf{\vdots}$}
\put(118,-11){$\bf{\vdots}$}
\put(168,-11){$\bf{\vdots}$}
\put(218,-11){$\bf{\vdots}$}

\put(18,245){$\bf{\vdots}$}
\put(68,245){$\bf{\vdots}$}
\put(118,245){$\bf{\vdots}$}
\put(168,245){$\bf{\vdots}$}
\put(218,245){$\bf{\vdots}$}

\put(0,72){$\bf{\ddots}$}
\put(0,22){$\bf{\ddots}$}
\put(0,222){$\bf{\ddots}$}
\put(0,172){$\bf{\ddots}$}
\put(50,222){$\bf{\ddots}$}
\put(100,222){$\bf{\ddots}$}
\put(150,222){$\bf{\ddots}$}
\put(200,222){$\bf{\ddots}$}

\put(20,120){\line(1,0){200}}

\put(120,120){\vector(1,0){28}}
\put(20,120){\vector(1,0){28}}
\put(70,120){\vector(1,0){28}}
\put(170,120){\vector(1,0){28}}

%\letvertex A=(10,10)
%\drawloop[l](A){$a$}

\put(20,0){\line(0,1){240}}
\put(70,0){\line(0,1){240}}
\put(120,0){\line(0,1){240}}
\put(170,0){\line(0,1){240}}
\put(220,0){\line(0,1){240}}

\put(20,20){\vector(0,1){27}}
\put(20,70){\vector(0,1){27}}
\put(20,120){\vector(0,1){27}}
\put(20,170){\vector(0,1){27}}

\put(70,20){\vector(0,1){27}}
\put(70,70){\vector(0,1){27}}
\put(70,120){\vector(0,1){27}}
\put(70,170){\vector(0,1){27}}

\put(120,20){\vector(0,1){27}}
\put(120,70){\vector(0,1){27}}
\put(120,120){\vector(0,1){27}}
\put(120,170){\vector(0,1){27}}

\put(170,20){\vector(0,1){27}}
\put(170,70){\vector(0,1){27}}
\put(170,120){\vector(0,1){27}}
\put(170,170){\vector(0,1){27}}

\put(220,20){\vector(0,1){27}}
\put(220,70){\vector(0,1){27}}
\put(220,120){\vector(0,1){27}}
\put(220,170){\vector(0,1){27}}
\put(120,120){\circle*{6}}
\put(120,70){\circle*{6}}
\put(120,20){\circle*{6}}
\put(120,170){\circle*{6}}
\put(120,220){\circle*{6}}

\put(170,120){\circle*{6}}
\put(170,70){\circle*{6}}
\put(170,20){\circle*{6}}
\put(170,170){\circle*{6}}
\put(170,220){\circle*{6}}

\put(220,120){\circle*{6}}
\put(220,70){\circle*{6}}
\put(220,20){\circle*{6}}
\put(220,170){\circle*{6}}
\put(220,220){\circle*{6}}

\put(20,120){\circle*{6}}
\put(20,70){\circle*{6}}
\put(20,20){\circle*{6}}
\put(20,170){\circle*{6}}
\put(20,220){\circle*{6}}

\put(70,120){\circle*{6}}
\put(70,70){\circle*{6}}
\put(70,20){\circle*{6}}
\put(70,170){\circle*{6}}
\put(70,220){\circle*{6}}

\put(122,110){\tiny $\bf{(0,{0})}$}
\put(172,110){\tiny $\bf{(1,{0})}$}
\put(222,110){\tiny $\bf{(2,{0})}$}
\put(72,110){\tiny $\bf{(-1,{0})}$}
\put(22,110){\tiny $\bf{(-2,{0})}$}

\put(122,60){\tiny $\bf{(0,{-1})}$}
\put(172,60){\tiny $\bf{(1,{-1})}$}
\put(222,60){\tiny $\bf{(2,{-1})}$}
\put(72,60){\tiny $\bf{(-1,{-1})}$}
\put(22,60){\tiny $\bf{(-2,{-1})}$}

\put(122,10){\tiny $\bf{(0,{-2})}$}
\put(172,10){\tiny $\bf{(1,{-2})}$}
\put(222,10){\tiny $\bf{(2,{-2})}$}
\put(72,10){\tiny $\bf{(-1,{-2})}$}
\put(22,10){\tiny $\bf{(-2,{-2})}$}

\put(122,160){\tiny $\bf{(0,{1})}$}
\put(172,160){\tiny $\bf{(1,{1})}$}
\put(222,160){\tiny $\bf{(2,{1})}$}
\put(72,160){\tiny $\bf{(-1,1)}$}
\put(22,160){\tiny $\bf{(-2,{1})}$}

\put(122,210){\tiny $\bf{(0,{2})}$}
\put(172,210){\tiny $\bf{(1,2)}$}
\put(222,210){\tiny $\bf{(2,{2})}$}
\put(72,210){\tiny $\bf{(-1,{2})}$}
\put(22,210){\tiny $\bf{(-2,{2})}$}

\put(142,124){$a$}
\put(92,124){$a$}
\put(42,124){$a$}
\put(192,124){$a$}

\put(73,138){$b$}
\put(23,138){$b$}
\put(123,138){$b$}
\put(173,138){$b$}
\put(223,138){$b$}

\put(73,188){$b$}
\put(23,188){$b$}
\put(123,188){$b$}
\put(173,188){$b$}
\put(223,188){$b$}

\put(73,38){$b$}
\put(23,38){$b$}
\put(123,38){$b$}
\put(173,38){$b$}
\put(223,38){$b$}

\put(73,88){$b$}
\put(23,88){$b$}
\put(123,88){$b$}
\put(173,88){$b$}
\put(223,88){$b$}

\end{picture}
\end{center}
\begin{center}
{\it Figure 3.}
\end{center}

Its vertex set is the set of integer points in the
plane. The edges labelled by $a$ are along the $x$-axis, from $(k,0)$ to 
$(k+1,0)$, and there is a loop with label $a$ at each point $(k,l)$ with 
$l \ne 0$. The edges labelled by $b$ are all the upward edges of the grid,
that is, all edges from $(k,l)$ to $(k,l+1)$, where $(k,l) \in \ZZ^2$.
To these, we have to add the oppositely oriented edges whose labels are the
respective inverses (in Figure 3, the oppositely oriented edges together with the
corresponding labels are omitted for simplicity). 
The comb lattice is clearly a context-free graph (tree).
\end{exa}

%%%% HO perso molto tempo a disegnarlo con autograph, senza successo 	

We proceed giving some simple examples. It is very easy to see that 
context-freeness is not ``transitive'' in the following sense:
if $(G,H)$ and $(H,K)$ are context-free (with $G,H$ finitely generated
and $K \le H \le G$) then in general $(G,K)$ will not be context-free.

\begin{exa}\label{exa:Z2} Let $G = \ZZ^2$, $H=\ZZ \times \{0\}\cong \ZZ$ and 
$K= \{(0,0)\}$. Then $H$ (i.e., $(H,K)$) is context-free. Of course, this also 
holds for $(G,H)$,  whose Schreier graphs are just the Cayley graphs of $\ZZ$. 
But $\ZZ^2$ (i.e., $(G,K)$) is not context-free. 
\end{exa}

This also shows that the reverse of Lemma \ref{lem:translator} does not hold
in general (while we know that it does hold when $[G:H] < \infty$).
Finally, we construct examples of three groups $K \le H \le G$, where
$(G,H)$ is context-free, $[H:K] < \infty$, and $(G,K)$ is not context-free.

\begin{exa}\label{exa:fin-ind}  
We construct a family of fully deterministic, symmetric labelled 
graphs $X_W\,$, $W \subset \ZZ$ (non-empty), and one such graph $Y$, so 
that $Y$ is the factor graph with respect to the action
of a 2-element group of automorphisms of each of the labelled graphs $X_W$. 
While $Y$ will be a context-free graph, many of the graphs $X_W$ in our family 
are not context-free. We then translate this back into the setting of pairs
of groups. 

The vertex set of $X_W$ is $\ZZ \times \{0,1\}$. The set of labels is 
$\Si = \{a, b, a^{-1}, b^{-1}\}$. The edges are as follows:
$$
\begin{aligned}
\bigl((k,0),a,(k+1,0)\bigr) \AND \bigl((k,1),a,(k+1,1)\bigr) \quad
&\text{for all}\; k \in \ZZ\,,\\ 
\bigl((k,0),b,(k+1,0)\bigr) \AND \bigl((k,1),b,(k+1,1)\bigr) \quad
&\text{for all}\; k \in \ZZ \setminus W\,,\AND\\ 
\bigl((k,0),b,(k+1,1)\bigr) \AND \bigl((k,1),b,(k+1,0)\bigr) \quad
&\text{for all}\; k \in  W\,.
\end{aligned}
$$
The reversed edges carry the respective inverse labels (in Figure 4, 
these reversed edges together with the corresponding labels are omitted for simplicity). 
Since $W \ne \emptyset$, there is at least one of the ``crosses''
(pair of the third type of edges). Therefore $X_W$ is connected.
In general, it does not have finitely many cone types, i.e., it is not
context-free. For example, it is not context-free 
when $W= \{k(|k|+1) : k \in \ZZ \}$

\begin{center}
\begin{picture}(330,130)
%\put(10,30){\line(1,0){340}}
%\put(10,80){\line(1,0){340}}

\put(30,30){\circle*{6}}
\put(80,30){\circle*{6}}
\put(130,30){\circle*{6}}
\put(180,30){\circle*{6}}
\put(230,30){\circle*{6}}
\put(280,30){\circle*{6}}
\put(330,30){\circle*{6}}

\put(30,80){\circle*{6}}
\put(80,80){\circle*{6}}
\put(130,80){\circle*{6}}
\put(180,80){\circle*{6}}
\put(230,80){\circle*{6}}
\put(280,80){\circle*{6}}
\put(330,80){\circle*{6}}

\put(15,15){\tiny $\bf{(-3,{1})}$}
\put(65,15){\tiny $\bf{(-2,{1})}$}
\put(115,15){\tiny $\bf{(-1,{1})}$}
\put(170,15){\tiny $\bf{(0,{1})}$}
\put(219,15){\tiny $\bf{(1,{1})}$}
\put(269,15){\tiny $\bf{(2,{1})}$}
\put(320,15){\tiny $\bf{(3,{1})}$}

\put(15,92){\tiny  $\bf{(-3,{0})}$}
\put(65,92){\tiny $\bf{(-2,{0})}$}
\put(115,92){\tiny $\bf{(-1,{0})}$}
\put(170,92){\tiny $\bf{(0,{0})}$}
\put(219,92){\tiny $\bf{(1,{0})}$}
\put(269,92){\tiny $\bf{(2,{0})}$}
\put(320,92){\tiny $\bf{(3,{0})}$}

\put(280,30){\vector(1,0){28}}
\put(180,30){\vector(1,0){28}}
\put(230,28){\vector(1,0){28}}
\put(230,32){\vector(1,0){28}}
\put(130,32){\vector(1,0){28}}
\put(130,28){\vector(1,0){28}}
\put(80,32){\vector(1,0){28}}
\put(80,28){\vector(1,0){28}}
\put(30,30){\vector(1,0){28}}

\put(280,30){\line(1,0){50}}
\put(180,30){\line(1,0){50}}
\put(230,28){\line(1,0){50}}
\put(230,32){\line(1,0){50}}
\put(130,32){\line(1,0){50}}
\put(130,28){\line(1,0){50}}
\put(80,32){\line(1,0){50}}
\put(80,28){\line(1,0){50}}
\put(30,30){\line(1,0){50}}

\put(280,80){\vector(1,0){28}}
\put(180,80){\vector(1,0){28}}
\put(230,78){\vector(1,0){28}}
\put(230,82){\vector(1,0){28}}
\put(130,82){\vector(1,0){28}}
\put(130,78){\vector(1,0){28}}
\put(80,82){\vector(1,0){28}}
\put(80,78){\vector(1,0){28}}
\put(30,80){\vector(1,0){28}}

\put(280,80){\line(1,0){50}}
\put(180,80){\line(1,0){50}}
\put(230,78){\line(1,0){50}}
\put(230,82){\line(1,0){50}}
\put(130,82){\line(1,0){50}}
\put(130,78){\line(1,0){50}}
\put(80,82){\line(1,0){50}}
\put(80,78){\line(1,0){50}}
\put(30,80){\line(1,0){50}}

\put(280,80){\vector(1,-1){35}}
\put(280,80){\line(1,-1){50}}
\put(180,80){\vector(1,-1){35}}
\put(180,80){\line(1,-1){50}}
\put(30,80){\vector(1,-1){35}}
\put(30,80){\line(1,-1){50}}

\put(280,30){\vector(1,1){18}}
\put(280,30){\line(1,1){50}}
\put(180,30){\vector(1,1){18}}
\put(180,30){\line(1,1){50}}
\put(30,30){\vector(1,1){18}}
\put(30,30){\line(1,1){50}}

\put(100,87){$a$}
\put(150,87){$a$}
\put(250,87){$a$}
\put(102,66){$b$}
\put(152,66){$b$}
\put(252,66){$b$}
\put(50,85){$a$}
\put(200,85){$a$}
\put(300,85){$a$}

\put(100,37){$a$}
\put(150,37){$a$}
\put(250,37){$a$}
\put(102,16){$b$}
\put(152,16){$b$}
\put(252,16){$b$}
\put(50,35){$a$}
\put(200,35){$a$}
\put(300,35){$a$}

\put(40,50){$b$}
\put(65,50){$b$}

\put(190,50){$b$}
\put(215,50){$b$}

\put(290,50){$b$}
\put(315,50){$b$}

\put(338,77){$\bf{\cdots}$}
\put(7,77){$\bf{\cdots}$}
\put(338,27){$\bf{\cdots}$}
\put(7,27){$\bf{\cdots}$}

\end{picture}
\end{center}
\begin{center}
{\it Figure 4.}
\end{center}
For arbitrary $W$, the two-element group that exchanges each $(k,0)$ with
$(k,1)$ acts on $X_W$ by label preserving graph automorphisms. 
The factor graph $Y$ (see Figure 5) has vertex set $\ZZ$ and edges 
$$
(k,a,k+1) \AND (k,b,k+1) \quad
\text{for all}\; k \in \ZZ\,,
$$
plus the associated reversed edges (in Figure 5, these edges together with the corresponding labels are
omitted for simplicity). It is clearly a context-free graph.
\begin{center}
\begin{picture}(330,60)
%\put(10,30){\line(1,0){340}}
%\put(10,80){\line(1,0){340}}

\put(30,30){\circle*{6}}
\put(80,30){\circle*{6}}
\put(130,30){\circle*{6}}
\put(180,30){\circle*{6}}
\put(230,30){\circle*{6}}
\put(280,30){\circle*{6}}
\put(330,30){\circle*{6}}

\put(15,15){\tiny $\bf{(-3,{1})}$}
\put(65,15){\tiny $\bf{(-2,{1})}$}
\put(115,15){\tiny $\bf{(-1,{1})}$}
\put(170,15){\tiny $\bf{(0,{1})}$}
\put(219,15){\tiny $\bf{(1,{1})}$}
\put(269,15){\tiny $\bf{(2,{1})}$}
\put(320,15){\tiny $\bf{(3,{1})}$}

\put(280,32){\vector(1,0){28}}
\put(280,28){\vector(1,0){28}}
\put(180,32){\vector(1,0){28}}
\put(180,28){\vector(1,0){28}}
\put(230,28){\vector(1,0){28}}
\put(230,32){\vector(1,0){28}}
\put(130,32){\vector(1,0){28}}
\put(130,28){\vector(1,0){28}}
\put(80,32){\vector(1,0){28}}
\put(80,28){\vector(1,0){28}}
\put(30,32){\vector(1,0){28}}
\put(30,28){\vector(1,0){28}}

\put(280,32){\line(1,0){50}}
\put(280,28){\line(1,0){50}}
\put(180,32){\line(1,0){50}}
\put(180,28){\line(1,0){50}}
\put(230,28){\line(1,0){50}}
\put(230,32){\line(1,0){50}}
\put(130,32){\line(1,0){50}}
\put(130,28){\line(1,0){50}}
\put(80,32){\line(1,0){50}}
\put(80,28){\line(1,0){50}}
\put(30,32){\line(1,0){50}}
\put(30,28){\line(1,0){50}}

\put(100,37){$a$}
\put(52,16){$b$}
\put(150,37){$a$}
\put(102,16){$b$}
\put(250,37){$a$}
\put(252,16){$b$}
\put(152,16){$b$}
\put(152,16){$b$}
\put(202,16){$b$}
\put(302,16){$b$}
\put(50,37){$a$}
\put(200,37){$a$}
\put(300,37){$a$}

\put(338,27){$\bf{\cdots}$}
\put(7,27){$\bf{\cdots}$}

\end{picture}
\end{center}
\begin{center}
{\it Figure 5.}
\end{center}
Now let $\F=\F_{\Si}$ be the free group (universal cover of $X_W$ and $Y$),
and for given  $W$, let $\K_W$ be the fundamental group of $X_W$ at
the vertex $(0,0)$. Furthermore, let $\K$ be the fundamental group of 
$Y$ at the vertex $0$. Then it is straightforward  that
$\K_W$ has index $2$ in $\K$. The mapping $\psi$ is the embedding of 
$\Si$ into $\F_{\Si}\,$, as above. We then have $Y= X(\F,\K,\psi)$ and 
$X_W = X(\F,\K_W,\psi)$, providing the required example.
\end{exa}

\end{document}